\documentclass[11pt]{article}
\usepackage[left=3.2cm,right=3.2cm,top=3.5cm,bottom=3cm]{geometry}
\usepackage[utf8]{inputenc}
\usepackage[english]{babel}
\usepackage{amsmath}
\usepackage{amssymb}
\usepackage{amsthm}
\usepackage{newpxtext,newpxmath}
\usepackage{braket}
\usepackage{amscd}
\usepackage{setspace}
\usepackage{booktabs}
\usepackage{lscape}
\usepackage{rotating}
\usepackage{graphicx}
\usepackage{tikz}
\usepackage{xcolor}
\usepackage{mathtools}
\usepackage{sidecap}
\usepackage{emptypage}
\usepackage{enumitem}
\usepackage{todonotes}
\usepackage[colorlinks=true, allcolors=blue, pagebackref=true]{hyperref}

\newtheorem{theorem}{Theorem}
\newtheorem*{theorem*}{Theorem}
\newtheorem{theoremi}{Theorem}

\newtheorem{lemma}{Lemma}
\newtheorem{proposition}{Proposition}
\newtheorem*{proposition*}{Proposition}
\newtheorem{corollary}{Corollary}
\newtheorem*{corollary*}{Corollary}

\newtheorem*{conjecture*}{Conjecture}

\newtheorem*{result*}{Result}
\theoremstyle{remark}
\newtheorem{remark}{Remark}
\theoremstyle{definition}
\newtheorem{definition}{Definition}
\newtheorem*{definition*}{Definition}
\newtheorem{example}{Example}
\newtheorem*{example*}{Example}

\newtheorem*{Ack}{Acknowledgments}

\newcommand{\Z}{\mathbb{Z}}
\newcommand{\C}{\mathbb{C}}
\newcommand{\Q}{\mathbb{Q}}
\newcommand{\R}{\mathbb{R}}
\newcommand{\F}{\mathbb{F}}
\newcommand{\pro}{\mathbb{P}}
\newcommand{\Po}{\mathbb{H}}

\newcommand{\fa}{\mathfrak{a}}

\newcommand{\spec}{\mathrm{spec}}

\newcommand{\SL}{\mathrm{SL}}

\newcommand{\fp}{\mathfrak{p}}

\def\={\;=\;}
\def\:={\;:=\;}
\def\+{\;+\;}
\def\-{\;-\;}

\title{Atkin polynomials for families of abelian varieties with real multiplication}
\author{Gabriele Bogo \and Yingkun Li}
\date{}

\begin{document}
\maketitle

\begin{abstract}
Generalizing the work of Atkin and Kaneko–Zagier in the elliptic case~\cite{KZ}, we describe the non-ordinary locus of a genus-zero non-compact curve~$Y$ in a Hilbert modular variety in terms of the zeros of generalized Atkin's orthogonal polynomials. The argument relies on the recent construction of lifts of partial Hasse invariants for~$Y$~\cite{BLpHi}.
We further describe these orthogonal polynomials as denominators of Padé approximants to the logarithmic derivatives of solutions of the Picard–Fuchs differential equations associated with~$Y$. This provides a new link between Padé approximation and the geometry of the non-ordinary locus, extending a classical observation of Igusa for the Legendre family and applying, in particular, to situations where the Picard–Fuchs equations do not admit modular solutions.
As applications, we determine the three-term recurrence relations for Atkin polynomials attached to triangle curves via hypergeometric identities, and compute the supersingular locus of a double cover of the Teichmüller curve~$W_{17}$. In the latter case, we conjecture that the associated supersingular polynomial is self-reciprocal, implying that supersingular points occur in pairs.
\end{abstract}

Partially building on unpublished work by Atkin, Kaneko and Zagier~\cite{KZ} discovered in 1997 a surprising connection between the $J$-invariants of supersingular elliptic curves and the zeros of certain orthogonal polynomials. More precisely, let
\begin{equation}
\label{eq:sspintro}
\mathrm{ss}_p(J)\=\prod_{\substack{E/\overline{\F}_p \\ E\text{ supersingular}}}(J-J(E))\,\quad\in\F_p[J]
\end{equation}
be the polynomial describing supersingularity of elliptic curves over~$\overline{\F}_p$ in terms of their~$J$-invariant. 

\begin{theorem*}[Atkin, Kaneko–Zagier]
  There exists a positive definite scalar product~$\langle\,,\rangle$ on~$\R[J]$ and a unique associated family of monic orthogonal polynomials~$\{A_n(J)\}_{n} \subset \Q[J]$
  with the following property: for every prime~$p>3$ there exists~$n_p\in\Z_{\ge0}$ such that ~$A_{n_p}(J)$ has~$p$-integral coefficients and
\[
A_{n_p}(J)\;\equiv\;\mathrm{ss}_p(J)\mod p\,.
\]
\end{theorem*}

The scalar product~$\langle\,,\rangle$ is called~\emph{Atkin's scalar product} and the polynomials~$\{A_n(J)\}_n$ are called~\emph{Atkin's polynomials}. 
Kaneko and Zagier proved that all Hecke operators are self-adjoint with respect to the Atkin scalar product, and gave closed-form expressions for the Atkin polynomials.
Their proof is based on the interpretation of the Eisenstein series~$E_{p-1}$ as a lift of the Hasse invariant, combined with the Padé approximation of the generating function of moments of the scalar product~$\langle\,,\rangle$. 

Several extensions of the work of Kaneko and Zagier have appeared over the years.
Some of these, imitating~\cite{KZ} and relying on the lifting of the Hasse invariant and the~$p$-integrality of the~$q$-expansion of modular forms, established Kaneko–Zagier-type results for the moduli space of elliptic curves with level structure, under a natural genus-zero assumption on the moduli space; see, for instance, the works of Sakai~\cite{Sa} and Tsutsumi~\cite{Ts}.
In a different direction, M.~Lane, in his Ph.D. thesis~\cite{Lane}, investigated Kaneko–Zagier-type results for Jacobians of hyperelliptic curves parametrized by Hecke triangle curves~$\Po/\Delta(2,m,\infty)$.
Recall that~$\Delta(n,m,\infty)$, for~$2\le n<m<\infty$, is a genus-zero Fuchsian group with one cusp and two elliptic points of orders~$n$ and~$m$, respectively; the most familiar example is~$\Delta(2,3,\infty)\simeq\SL_2(\Z)$, but for general~$n$ and~$m$, the group~$\Delta(n,m,\infty)$ is non-arithmetic.
The algebraic structure of the space of modular forms on~$\Delta(2,m,\infty)$ is analogous to that for~$\SL_2(\Z)$ (this is used in the definition and computation of Atkin polynomials).
However, due to the lack of a theory of Hasse invariants for triangle curves, Lane established the relation between the non-ordinary reduction of~$\Po/\Delta(2,m,\infty)$ modulo~$p$ and Atkin-like polynomials by explicitly computing the corresponding Hasse–Witt matrices.

We present here an approach that both unifies and vastly generalizes the result of Kaneko and Zagier, as well as its subsequent extensions, by considering families of abelian varieties arising from curves in Hilbert modular varieties.
Let~$F$ be a totally real field of degree~$g\ge2$ over~$\Q$, and let~$\mathcal{M}_F$ denote the Hilbert modular scheme associated with~$F$.
Let~$Y\hookrightarrow M_F  := \mathcal{M}_F(\C)$ be a smooth algebraic curve embedded in the complex points of the Hilbert modular variety.
A result of Möller–Viehweg~\cite{MV} implies that, if~$Y$ is affine, it admits an integral model~
$\mathcal{Y}$ defined over~$\mathcal{O}_K[S^{-1}]$,
where~$K$ is a number field, $\mathcal{O}_K$ is its ring of integers, and~$S$ is a finite set of rational primes.
In addition, there is a smooth family of abelian varieties~$\mathcal{X}\to\mathcal{Y}$ with real multiplication by~$F$.
For the sake of exposition, we assume throughout the introduction that~$K=\Q$, in order to deal with polynomials defined over the residue field~$\F_p$. The constructions of the paper work, however, for any field~$K$ and~$p\not\in S$, with~$\F_p$ replaced by the collection of residue fields~$\mathcal{O}_K[S^{-1}]\otimes\F_p$. The existence of generalized Atkin polynomials, related to the positive definiteness of the generalized scalar products, is proven for fields~$K\subset\R$; though it may hold more generally, it is known to fail in some cases~\cite{Ts} (see the beginning of Section~\ref{sec:posdef}).

Assume that~$Y$ is a non-compact curve of genus zero in the Hilbert modular variety~$M_F$
The associated integral model is~
$$
\mathcal{X}\to\mathcal{Y}\simeq\mathrm{spec}\left(\Z[S^{-1}]\left[J,\prod_{i=1}^s(J-J_i)^{-1}\right]\right),$$
for some~$J_i\in\Z[S^{-1}]$.
Let~$p\notin S$ be a prime unramified in $F$ and consider the base change~$\mathcal{X}_{\overline{p}}\to\mathrm{spec}\bigl(\overline{\F}_p[S^{-1}][J,\prod_{i=1}^s(J-J_i)^{-1}]\bigr)$, where we assume that~$J_1,\dots,J_s$ remain distinct over~${\F}_p$. Denote by~$X_{\overline{p},J_0}$ the fiber of~$\mathcal{X}_{\overline{p}}$ over~$J_0\in\overline{\F}_p\smallsetminus\{J_1,\dots,J_s\}$. 
The analogue in this setting of the supersingular polynomial~\eqref{eq:sspintro} is given by
\begin{equation}
\label{eq:nopintro}
\mathrm{no}_\fp(J):=\!\!\!\!\!\prod_{\substack{J_0\in\overline{\F}_p \\ X_{\overline{p},J_0}\,\text{non-ordinary}}}\!\!\!\!\!{(J-J_0)}\,,
\end{equation}
the polynomial describing the non-ordinary locus of~$\mathcal{X}_{\overline{p}}$ in terms of~$J$ (see Section~\ref{sec:Koba} for details). Our first result can be formulated as follows.

\begin{theoremi}[Theorem~\ref{thm:main2}]
\label{thm:intro2}
Let~$Y\simeq\Po/\Gamma\hookrightarrow M_F$ be a non-compact genus-zero curve in a Hilbert modular variety with integral model~$\mathcal{X}\to\mathcal{Y}$ defined over~$\Z[S^{-1}]$. 
There exists a scalar product~$\langle\,,\rangle$ on~$\R[J]$ and, if positive definite, a unique family of monic orthogonal polynomials~$\{A_{n}(J)\}\subset \Q[J]$ with the following property: for every prime~$p\not \in S$ there exists~$n_p\in\Z_{\ge0}$ such that~$A_{n_p}(J)$ has~$p$-integral coefficients and 
\[
A_{n_p}(J)\;\equiv\;\mathrm{no}_p(J)\mod p\,.
\]
\end{theoremi}

The definition of the scalar product appearing in Theorem~\ref{thm:intro2}, given in Section~\ref{sec:main}, relies on the notion of twisted modular form introduced by Möller and Zagier in~\cite{MZ}, as well as on the algebraic structure of the space of twisted modular forms studied in~\cite{BPf} (see Section~2 for a brief review). Theorem~\ref{thm:intro2} is proven by lifting the Hilbert Hasse invariant to a Hilbert modular form in characteristic~zero, restricting it to a (twisted) modular form on the curve~$Y$, and then relating the zeros of this (twisted) modular form to orthogonal polynomials via Padé approximants to the generating function of moments of the scalar product~$\langle\,,\rangle$, in analogy with Kaneko and Zagier’s proof. Details can be found in Section~\ref{sec:Hasse}.

A few remarks are in order.
First, for~$g=2$, the result can be refined to describe the supersingular locus of~$\mathcal{X}_{\overline{p}}$. In this way, the results of Kaneko and Zagier, as well as their extensions to congruence subgroups, are recovered as special cases by viewing modular curves as Hirzebruch–Zagier curves inside Hilbert modular surfaces (see section \ref{sec:examples}).
Second, Theorem~\ref{thm:intro2} applies to infinitely many new situations—for instance, to all non-compact triangle curves and to all genus-zero Teichmüller curves in Hilbert modular surfaces.
Third, the genus-zero assumption on the base curve~$Y$ is necessary in order to express the entire non-ordinary locus as the zero set of a polynomial in a modular function. In the higher-genus case, one could instead work with polynomials in suitable modular functions on affine charts and then glue the resulting local data together. We did not pursue this approach.

While it is appealing to relate the non-ordinary locus of a curve in a Hilbert modular variety to a single family of orthogonal polynomials, Theorem~\ref{thm:intro2} has one drawback: the number~${n_p}$ appearing in its statement, which is the cardinality of the non-ordinary locus of~$\mathcal{X}_{\overline{p}}$, is at the current stage not directly computable. Upper and lower bounds for $n_p$ can be computed (see Corollary 4 of~\cite{BPf}, and Section 2.4 of~\cite{BLpHi} for explicit computations), but for the scope of Theorem~\ref{thm:intro2} the exact value is needed. 
To overcome this difficulty, we describe each component of the non-ordinary locus of~$\mathcal{X}_{\overline{p}}$ in terms of a family of orthogonal polynomials.

If~$F$ is a totally real field of degree~$g\ge2$, the non-ordinary locus of the Hilbert modular scheme~$\mathcal{M}_{F,s}$ over~$s=\mathrm{Spec}({\overline{\F}_p})$, with~$p$ unramified in~$F$, has~$g$ distinct components~$D_1,\dots,D_g$.
For a curve~$Y\hookrightarrow M_F$ with model over~$\Z[S^{-1}]$ and a prime~$p\not\in S$, let~$\mathcal{X}_{\overline{p}}$ be as before~\eqref{eq:nopintro}. For~$j=1,\dots,g$ the polynomials
\[
\mathrm{ph}_{p,j}(J):=\prod_{\substack{J_0\in\overline{\F}_p \\ X_{\overline{p},J_0}\in D_j}}(J-J_0)\,,\quad j=1,\dots,g\,,
\] 
refine the polynomial~\eqref{eq:nopintro} by describing each component of the non-ordinary locus of~$\mathcal{X}_{\overline{p}}$.
One of our main result expresses these polynomials in terms of orthogonal polynomials.

\begin{theoremi}[Theorem~\ref{thm:main}, Corollary~\ref{cor:main}]
\label{thm:intro1}
Let~$Y\simeq\Po/\Gamma\hookrightarrow M_F$ be a non-compact genus-zero curve in a Hilbert modular variety with integral model~$\mathcal{X}\to\mathcal{Y}$ defined over~$\Z[S^{-1}]$. 
\begin{enumerate}
\item For~$j=1,\dots,g$ there exist a scalar product~$\langle\,,\rangle_j$ on~$\R[J]$ and, if positive definite, a unique family of monic orthogonal polynomials~$\{A_{n,j}(J)\}_n\subset \Q[J]$ with the following property: for every prime~$p\not \in S$ there exists~$n_{p,j}\in\Z_{\ge0}$ such that~$A_{n_{p,j},j}(J)$ has~$p$-integral coefficients and
\[
A_{n_{p,j},j}(J)\;\equiv\;\mathrm{ph}_{p,j}(J)\mod p\,.
\]
The index~$n_{p,j}$ is explicitly computable as the dimension of a space of (twisted) modular forms. 
\item For every prime~$p\not\in S$, and~$n_{p,1},\dots,n_{p,g}\in\Z_{\ge0}$ as above, it holds
\[
\begin{aligned}
\mathrm{lcm}(A_{n_{p,1},1}(J),\dots,A_{n_{p,g},g}(J))&\;\equiv\;\mathrm{no}_p(J)\mod p\,.
\end{aligned}
\]
\end{enumerate}
\end{theoremi}

Similarly to Theorem~\ref{thm:intro2}, the scalar products are defined in terms of twisted modular forms. 
In Proposition~1 we show that these scalar products satisfy the analogue of the positive definiteness condition whenever $Y$ is defined over~$\R$. This condition is satisfied for instance by all triangle curves and Teichmüller curves in genus two. 

The proof of Theorem~\ref{thm:intro1} is based on the properties of certain special twisted modular forms~$h_{p,1},\dots,h_{p,g}$, which are analogues, for genus-zero curves~$Y\hookrightarrow M_F$, of the lift of the Hasse invariant in the elliptic case.
These twisted modular forms were constructed in~\cite{BLpHi}, and their properties are recalled in Section~\ref{sec:pHi}.
The notions of Serre derivative and quasimodular form in the twisted setting, which are needed in the proof of Theorem~\ref{thm:intro1}, are introduced in Section~\ref{sec:Serre}.
The proof itself, presented in Section~\ref{sec:ssl}, combines these new ingredients with a Padé-type approximation of the generating function of moments of the scalar products, in analogy with the original proof of Kaneko and Zagier~\cite{KZ}.

In Section~\ref{sec:PF} we generalize Theorem~\ref{thm:intro1} by interpreting the Atkin polynomials in terms of solutions of Picard–Fuchs differential equations. To a genus-zero curve in a~$g$-dimensional Hilbert modular variety with a model over~$R=\mathcal{O}_K[S^{-1}]$ are attached~$g$ second-order (Picard–Fuchs) linear differential operators~$L_1,\dots,L_g$.
It was shown in~\cite{BPf} that the normalized holomorphic solution~$y_j(t)=1+O(t)$ of~$L_jv=0$ at a point of maximal unipotent monodromy has a power series expansion with coefficients in~$R$. We prove that Atkin polynomials~$\{A_{j,n}(J)\}_n$ essentially arise as denominators of Padé approximants of the logarithmic derivative of the integral solution~$y_j(t)$.

\begin{theoremi}
\label{thm:noPF}
Let~$Y\simeq\Po/\Gamma\hookrightarrow M_F$ be a genus-zero non-compact curve in a Hilbert modular variety, with integral model~$\mathcal{X}\to\mathcal{Y}$ defined over $\Z[S^{-1}]$. 
Let~$y_j(t)\in 1+\Z[S^{-1}][\![t]\!]$ denote the normalized integral solution of the~$j$-th Picard–Fuchs differential operator of~$Y$, and let~$J:=t^{-1}$. Then, for every prime~$p\not\in S$, the denominator of the~$[n_{p,j}-1,n_{p,j}]$-Padé approximant of the power series
\[
\Phi_{j}(t)\:= t-\frac{2\cdot t^{2}}{\chi(Y)\cdot\lambda_{j}\cdot N}\frac{y_{j}'(t)}{y_{j}(t)}
\]
is congruent to~$\mathrm{ph}_{p,j}(J)$ modulo~$p$, where~$n_{p,j}$ is the degree of~$\mathrm{ph}_{p,j}$. Here~$\chi(Y)$ denotes the orbifold Euler characteristic of~$Y$, and~$\lambda_1,\dots,\lambda_g\in(0,1]\cap\Q$ its Lyapunov exponents, and~$N$ is the least common multiple of the orders of the elliptic points of~$\Gamma$ (set $N=1$ for~$\Gamma$ torsion-free). 

Moreover, the family of denominators of~$[n-1,n]$-Padé approximants of~$\Phi_j(t)$ is the family of Atkin polynomials~$A_{j,n}(J)$, when the latter is defined. 
\end{theoremi}

A connection between the Hasse invariant of elliptic curves and solutions of Picard–Fuchs equations was classically discovered by Igusa~\cite{Ig2}, and generalized to partial Hasse invariants of certain Kobayashi geodesics in~\cite{BPf}. In these cases, the partial Hasse polynomial is obtained by~\emph{truncating} the integral solution and reducing it modulo~$p$. As explained in~\cite{BPf}, however, this method works only when the partial Hasse polynomial has degree smaller than~$p$. 
Theorem~\ref{thm:noPF} generalizes the result of Igusa and of~\cite{BPf}, since in those cases the denominator of the Padé approximant matches the truncation of the solution. Moreover, it applies without any restriction on the cardinality of the non-ordinary locus; it permits an easy computation of the partial Hasse polynomials of curves for which no method was known before, and it is well-suited for implementation in common computer algebra systems. More importantly, no assumption on modularity is made. This means that this result can be extended to Picard–Fuchs equations of one-parameter families of abelian varieties that are not necessarily modular (see for example the Picard-Fuchs equations recently studied by van Hoeij, van Straten, and Zudilin~\cite{vHvSZ}, whose modular properties are unclear).  

The proof of Theorem~\ref{thm:noPF} relies on the integrality of solutions of the Picard–Fuchs equations, and on the fact that their modulo~$p$ reduction carry arithmetic information (they are related to entries of the Hasse–Witt matrix of the underlying family of abelian varieties). The relation between denominators of Padé approximants and orthogonal polynomials is classical and briefly recalled in Section~\ref{sec:ortho}. 

Finally, in Section~\ref{sec:examples} we explicitly compute the three-term recurrence for Atkin's polynomials on triangle curves using Gauss's contiguous relations, and we determine the supersingular locus of the curve~$\Po/\Delta(2,5,\infty)$.
We then apply Theorem~\ref{thm:noPF} to compute the non-ordinary locus of a degree-two cover of the Teichmüller curve~$W_{17}$ in the Hilbert modular surface~$\Po^2/\SL(\mathcal{O}_{17}\oplus\mathcal{O}^{\vee}_{17})$; the curve~$W_{17}$ is also defined over~$\Q(\sqrt{17})$.
This locus cannot be obtained using the truncation method described above.
Interestingly, in the cases we computed, the partial Hasse polynomials are self-reciprocal. We conjecture that this property holds in general.

\begin{conjecture*}
For~$j\in\{1,2\}$ and every prime~$p\ne2,17$, the partial Hasse polynomials of the double cover~$\Pi_{17}$ of the Teichmüller curve~$W_{17}\hookrightarrow\Po^2/\SL(\mathcal{O}_{17}\oplus\mathcal{O}^{\vee}_{17})$ satisfy
\[
\mathrm{ph}^{\Pi_{17}}_{p,j}(J)\=\mathrm{ph}^{\Pi_{17}}_{p,j}(J^{-1})\cdot J^{n_{p,j}}\,,
\]
where~$n_{p,j}$ is the degree of~$\mathrm{ph}_{p,j}$. 
\end{conjecture*}

The conjecture, verified for all primes~$p\le100$, implies in particular that the supersingular points on~$\Pi_{17}$ occur in pairs, namely~$J_0$ and~$J_0^{-1}$.
This phenomenon also appears for the supersingular (Hasse) polynomial of the Legendre family, as its classical explicit expression shows.
In contrast, it fails for the supersingular polynomial of~$\Po/\SL_2(\Z)$ and for the partial Hasse polynomials associated with~$\Po/\Delta(2,5,\infty)$; a notable difference is that these latter cases contain elliptic points and only one cusp. 

\begin{Ack}
The first author thanks Wadim Zudilin for valuable discussions and encouragement over the years. 
  G.\,Bogo is funded by the Deutsche Forschungsgemeinschaft (DFG, German Research Foundation) — SFB-TRR 358/1 2023 — 491392403.
  Y.\,Li is supported by the Heisenberg Program of the DFG, project number 539345613.
\end{Ack}

\section{Families of abelian varieties with real multiplication}
\label{sec:Koba}
\renewcommand{\thetheorem}{\arabic{theorem}}
\setcounter{theorem}{0}

Let~$F$ be a totally real number field of degree~$g$ over~$\Q$ and denote by~$\mathcal{M}_F\to\mathrm{spec}(\Z)$ the Hilbert modular variety associated to~$F$. It is a (coarse) moduli scheme of isomorphism classes of~$g$-dimensional polarized abelian varieties with real multiplication by~$\mathcal{O}_F$, the ring of integers of~$F$. 
The complex points of~$\mathcal{M}_F$ are the disjoint union of components~$\Po^g/\SL(\mathcal{O}_F\oplus\fa)$, where~$\fa$ runs over the components of the strict class group of~$F$. For the construction of~$\mathcal{M}_F$ and its properties, the reader may consult Chapter~3 of~\cite{GoBook}.  
In the following, we denote by~$M_F$ the complex points of~$\mathcal{M}_F$ and, for~$\overline{\F}_p$ a fixed algebraic closure of~$\F_p$, by~$\mathcal{M}_{F,s}$ the special fiber of~$\mathcal{M}_F$ over~$s=\mathrm{spec}(\overline{\F}_p)$ for~$p$ a prime unramified in~$F$.  

In order to describe the geometry of the special fiber~$\mathcal{M}_{F,s}$ we recall the notion of non-ordinary abelian variety. 
Let~$k$ be a field of characteristic~$p$ and let~$\overline{k}$ be an algebraic closure of~$k$. The~\emph{$p$-rank} of a~$g$-dimensional abelian variety~$A/k$ is the rank of the group of~$p$-torsion points of~$A(\overline{k})$. Recall that the abelian variety~$A/k$ is called~\emph{ordinary} if its~$p$-rank is maximal, i.e., if~$A[p](\overline{k})\simeq(\Z/p\Z)^g$; otherwise it is called~\emph{non-ordinary}. The abelian variety~$A/k$ is~\emph{supersingular} if it is isogenous over~$k$ to a product of~$g$ supersingular elliptic curves; if this isogeny is an isomorphism, one says that~$A/k$ is~\emph{superspecial}. 

The locus of ordinary abelian varieties in~$\mathcal{M}_{F,s}$ at~$s=\mathrm{Spec}({\overline{\F}_p})$ is a dense subscheme, and its complement is called the~\emph{non-ordinary locus}.
The non-ordinary locus in~$\mathcal{M}_{F,s}$ can be described in terms of divisors of certain Hilbert modular forms called~\emph{partial Hasse invariants}.
\begin{theorem*}[Goren~\cite{GHi}, Andreatta-Goren~\cite{AG}]
\begin{enumerate}[wide=0pt]
Let~$F$ as above and~$p$ be a prime unramified in~$F$.
\item The non-ordinary locus in~$\mathcal{M}_{F,s}$ has~$g$ components, denoted by~$D_j$ for~$j=1,\dots,g$. There exist~$g$ Hilbert modular forms~$h_1,\dots,h_g$ of non-parallel weight,  called~\emph{partial Hasse invariants}, such that~$\mathrm{div}(h_j)=D_j$. The partial Hasse invariants can be defined over a suitable finite field~$\F$.
\item For every~$j=1,\dots,g$, the~$q$-expansion of the partial Hasse invariant~$h_j$ at every cusp is constant and equal to~$1$.
\end{enumerate}
\end{theorem*}
The partial Hasse invariants are a natural generalization of the Hasse invariant of elliptic curves. Nevertheless, it is important to emphasize that the partial Hasse invariants~$h_i$, if~$\dim{\mathcal{M}_F}>1$ , are genuinely characteristic~$p$ objects, i.e., they do not extend to the integral model $\mathcal{M}_F$.
However, the~\emph{Hasse invariant}~$h$, defined as~$h:=\prod_{j=1}^gh_{j}$, is a Hilbert modular form of parallel weight~$(p-1)$ that admits a lift to a Hilbert modular form in characteristic zero. It follows from its description as a product of partial Hasse invariants that its divisor is supported at the non-ordinary locus of~$\mathcal{M}_{F,s}$ and that its~$q$-expansion is $1$ at every cusp. If~$g=2$ and the field~$F$ has class number~$1$, the Hasse invariant sometimes lifts to a Hilbert Eisenstein series of parallel weight~$(p-1)$ (see Section 14 of~\cite{AG}).

\medskip

We focus now on the theory for curves in Hilbert modular varieties. 
Let~$Y\hookrightarrow M_F$ be a smooth algebraic curve and assume that~$Y$ is affine. A result of Möller and Viehweg (Corollary 6.2 in~\cite{MV}) affirms that $Y\hookrightarrow{M}_F$ can be defined over a number field~$K$. If this is the case, by taking the Zariski closure of~$Y$ in the integral structure of~$\mathcal{M}_F$ along the embedding, one can endow~$Y$ with an integral structure $\mathcal{Y}$ over~$R:=\mathcal{O}_K[S^{-1}]$, where~$\mathcal{O}_K$ is the ring of integers of~$K$ and~$S\subset R$ is a finite set of primes. 
Given the moduli interpretation of~$\mathcal{M}_F$, the integral model is described by a family of polarized abelian varieties~$\mathcal{X}\to\mathcal{Y}$ with real multiplication by~$F$ defined over~$R$. In the following, we assume that if $\fp\in S$, then all its conjugates are in~$S$ or, equivalently, that $S$ comes from a set of rational primes. Then, via the map $\mathrm{Spec}(R)\to\mathrm{Spec}(\Z[S^{-1}])$, we can, and sometimes do, consider $\mathcal{X}$ as a scheme over $\Z[S^{-1}]$. 
In the paper we assume that all the primes that ramify in~$K$ and in~$F$ are in~$S$, and make more specific assumptions later. 

In the case~$Y$ is of genus zero and non compact, the associated integral model is of the form
\begin{equation}
\label{eq:model}
\mathcal{X}\to\spec\biggl(R\biggl[J,\frac{1}{\prod_{i=1}^s(J-J_i)}\biggr]\biggr)\,,
\end{equation}
for some~$J_1,\dots,J_s\in R$.

Fix a rational prime~$p\not\in S$ and an algebraic closure~$\overline{\F}_p$ of~$\F_p$.  Let~$\fp\subset\mathcal{O}_{K}$ be a prime over~$p$ and~$k_{\fp}:=\mathcal{O}_{K}/\fp\mathcal{O}_{K}$ the residue field. 
Consider the base change of the integral model~\eqref{eq:model} to the algebraic closure of the redisue field
\begin{equation}
\label{eq:modelp}
\mathcal{X}_{\bar{\fp}}:=\mathcal{X}\times_{\mathrm{spec}(R[J,\prod_{i=1}^s(J-J_i)^{-1}])}\spec\biggl(\overline{\F}_p\biggl[J,\frac{1}{\prod_{i=1}^s(J-J_i)}\biggr]\biggr)\,,
\end{equation}
where we assume that~$J_1,\dots,J_g$ stay distinct modulo~$\fp$. 
Denote by $\mathcal{X}_{\bar{\fp},J_{0}}$ the fiber of $\mathcal{X}_{\bar{\fp}}$ over $J=J_{0}\in\overline{\F}_p$.  
Since $\mathcal{X}_{\bar{\fp}}$ comes from reduction of a non-trivial family over $R$, for generic~$p$ only finitely many of its points are non-ordinary abelian varieties. Therefore, for such primes~$p$, the non-ordinary, supersingular, and superspecial locus of $\mathcal{X}_{\bar{\fp}}$  can be described via polynomials with coefficients in $\overline{\F}_p$:  
\[
\mathrm{no}_\fp(J):=\!\!\!\!\!\prod_{\overset{J_0\in\overline{\F}_p}{\mathcal{X}_{\bar{\fp},J_0}\text{non-ordinary}}}\!\!\!\!\!{(J-J_0)}\,,\quad
\mathrm{ss}_\fp(J):=\!\!\!\!\!\prod_{\overset{J_0\in\overline{\F}_p}{\mathcal{X}_{\bar{\fp},J_0}\text{supersingular}}}\!\!\!\!\!{(J-J_0)}\,,\quad
\mathrm{sp}_\fp(J):=\!\!\!\!\!\prod_{\overset{J_0\in\overline{\F}_p}{\mathcal{X}_{\bar{\fp},J_0}\text{superspecial}}}\!\!\!\!\!{(J-J_0)}\,.
\]
By means of the identification~$R\otimes\overline{\F}_p\simeq \prod_{\fp|p}\overline{\F}_p$, we define the~\emph{non-ordinary, supersingular}, and \emph{superspecial polynomials} associated to the modulo~$p$ reduction of~$Y$ respectively by 
\[
\mathrm{no}_p(J):=\bigl(\mathrm{no}_\fp(J)\bigr)_{\fp|p}\,,\quad \mathrm{ss}_p(J):=\bigl(\mathrm{ss}_\fp(J)\bigr)_{\fp|p}\,,\quad \mathrm{sp}_p(J):=\bigl(\mathrm{sp}_\fp(J)\bigr)_{\fp|p}\quad\in \prod_{{\fp|p}}\overline{\F}_p[J]\,.
\]

As recalled above, the non-ordinary locus in~$\mathcal{M}_{F,s}$ is described by the union of divisors~$D_1,\dots,D_g$ of the partial Hasse invariants. For~$\fp\subset\mathcal{O}_K$ a prime ideal over~$p$, we refine the polynomial~$\mathrm{no}_{\fp}(J)$ by considering the fibers of the family~$\mathcal{X}_{\bar{\fp}}$ that lie on the divisor~$D_j$
\[
\mathrm{ph}_{\fp,j}(J):=\prod_{\overset{J_0\in\overline{\F}_p}{\mathcal{X}_{\bar{\fp},J_0}\in D_j}}(J-J_0)\,,\quad j=1,\dots,g\,.
\]
For~$j=1,\dots,g$, the~\emph{partial Hasse polynomial} 
\begin{equation}
\label{eq:phpdef}
\mathrm{ph}_{p,j}(J)\:=\bigl(\mathrm{ph}_{\fp,j}(J)\bigr)_{\fp|p}\quad\in \prod_{\fp|p}\overline{\F}_p[J]\,,
\end{equation}
is the polynomial describing the intersection of~$\mathcal{X}_{\bar{\fp}}$ with the divisor~$D_j$ of the partial Hasse invariant~$h_j$.

\section{Twisted modular forms on genus zero curves}
\label{sec:tmf0}

Let~$F\subset \R$ be a totally real number field of degree~$g$ over~$\Q$, and let~$\sigma_1,\dots,\sigma_g\colon F\to\R$ denote the real embeddings, where~$\sigma_1$ is the identity embedding. 
Let~$Y\simeq\Po/\Gamma\hookrightarrow M_F$ be a smooth algebraic curve in a Hilbert modular variety. If~$Y$ is irreducible, it is mapped into a component~$\Po^g/\SL(\mathcal{O}_F\oplus\fa)$ of~$M_F$, where~$\fa$ is a fractional ideal of~$\mathcal{O}_F$, and the embedding has the following description in terms of universal covering spaces. There exists an inclusion of groups~$\Gamma\hookrightarrow\SL(\mathcal{O}_F\oplus\fa)$ and a collection~$\varphi=(\varphi_1,\dots,\varphi_g)$ of holomorphic maps~$\varphi_j\colon\Po\to\Po$, called~\emph{modular embedding}, such that for every~$\gamma\in\Gamma$ it holds
\begin{equation}
\label{eq:modem}
\varphi_j(\gamma\tau)\=\gamma^{\sigma_j}\varphi_j(\tau)\,,\quad j=1,\dots,g\,,
\end{equation}
where~$\gamma^{\sigma_j}$ denotes the image of~$\gamma\in\Gamma\subset\SL(\mathcal{O}_F\oplus\fa)$ via the real embedding~$\sigma_j$ (applied entry-wise), and~$\Gamma$ acts by Möbius transformations both on domain and range. It is possible to normalize~$\varphi_1(\tau)=\tau$, and we do so in the following.

Let~$\vec{k}=(k_1,\dots,k_g)\in\Q^g$. A~\emph{twisted modular form} of weight~$\vec{k}$ and multiplier system~$v\colon\Gamma\to\C^\times$ is a holomorphic function~$f\colon\Po\to\C$ such that, for every~$\gamma=\left(\begin{smallmatrix}a & b\\c & d\end{smallmatrix}\right)\in\Gamma$, it holds
\[
f(\gamma\tau)\=v(\gamma)\cdot f(\tau)\cdot\prod_{j=1}^g(\sigma_j(c)\varphi_j(\tau)+\sigma_j(d))^{k_j}\,,
\]
where we take the principal value for the branch of the rational powers. 
Growth conditions at the cusps of~$\Gamma$ are also imposed; see Section 1.1.4 of~\cite{BPf}. Denote by~$M_{\vec{k}}(\Gamma,\varphi,v)$ the space of twisted modular forms on~$(\Gamma,\varphi)$, where~$\varphi$ is the modular embedding, with weight~$\vec{k}$ and multiplier system~$v$. 

\begin{example}
Classical modular forms of weight~$k$ are twisted modular forms of weight~$\vec{k}=(k,0,\dots,0)$. A non-classical example is given by the derivative~$\tfrac{d}{d\tau}\varphi_j(\tau)$ of a component of the modular embedding; it is a twisted modular form of weight~$(2,0,\dots,0,-2,0,\dots,0)$, where~$-2$ is in the~$j$-th entry. The modular property is easily deduced from~\eqref{eq:modem}, and the growth at the cusp has been studied in Section 2 of~\cite{MZ}. Further non-trivial examples are given in Section~\ref{sec:pHi}.
\end{example}

While twisted modular forms admit Fourier expansion at the cusps, in general the Fourier coefficients do not have nice~$p$-integrality properties. This is because the groups~$\Gamma$ admitting a modular embedding are in general not arithmetic groups. We consider instead another expansion at the cusp, which is more useful for our purposes. 

Assume that~$\Gamma$ has genus zero and let~$t$ be a Hauptmodul realizing the isomorphism~$\Po/\Gamma\simeq Y$ of the quotient space with a punctured sphere.
The punctures are the images, via~$t$, of the cusps and elliptic points of~$\Gamma$. The result of Möller–Viehweg (Corollary 6.2 in~\cite{MV})  mentioned in Section~\ref{sec:Koba} implies that the set of punctures can be defined over a number field~$K$. Moreover, if $Y$ admits an integral model $\mathcal{Y}$ over~$R=\mathcal{O}_K[S^{-1}]$ one can, up to enlarging $S$, assume that the values of~$t$ at the cusps and elliptic points are defined over~$R$. Normalize~$t$ further to have a (simple) zero at the cusp~$i\infty$; every twisted modular form~$f$ can then be expressed locally near~$i\infty$ as a power series in~$t$. We call this the~\emph{$t$-expansion} of~$f$ at~$i\infty$.

\begin{theorem}[Proposition 3 and Theorem 2 in~\cite{BPf}, Theorem 1 in~\cite{BLpHi}]
\label{thm:Bj}
Let~$\Gamma$ be a genus-zero group with modular embedding~$\varphi$, and let~$t$ be the Hauptmodul introduced above. 
\begin{enumerate}
\item There exist twisted modular forms~$B_2,\dots,B_g$, with non-trivial multiplier system, that have no zeros in~$\Po\cup\{\infty\}$.
 For $j\ge2$, the weight of~$B_j$ is~$(-\lambda_j,0,\dots,0,1,0,\dots,0)$, where~$1$ is in the~$j$-th entry, and~$\lambda_j\in(0,1]\cap\Q$ is defined by
\begin{equation}
\label{eq:Lyap}
\deg\mathrm{div}(\varphi_j')\=(1-\lambda_j)\chi(\Gamma)\,,
\end{equation}
where~$\chi(\Gamma)$ denotes the orbifold Euler characteristic of~$\Gamma$. Moreover, $B_2,\dots,B_g$ can be normalized to have~$p$-integral $t$-expansion at the cusp~$i\infty$ for every prime~$p\not\in S$. 
\item If~$\vec{k}=(k_1,\dots,k_g)\in\Q\times\Z^{g-1}$ and~$f\in M_{\vec{k}}(\Gamma,\varphi,v)$, there exists a classical modular form~$f_1$ of rational weight and non-trivial multiplier system such that~$f=f_1\cdot\prod_{j=2}^gB_j^{k_j}$.
\item Let~$\vec{k}=(k_1,\dots,k_g)\in\Q\times\Z^{g-1}$. For every multiplier system~$v$, the space~$M_{\vec{k}}(\Gamma,\varphi,v)$ admits a basis of twisted modular forms whose $t$-expansion at~$i\infty$ has algebraic~$p$-integral coefficients for every prime~$p\not\in S$. 
\end{enumerate}
\end{theorem}

\begin{remark}
The last statement in the theorem holds more generally for the expansion at every cusp, after a suitable choice of Hauptmodul. 
\end{remark}

Part 2 of the theorem states that twisted modular forms on genus-zero groups can be studied in terms of classical modular forms and of the twisted modular forms~$B_2,\dots,B_g$. In the case~$\Gamma$ has genus zero, the space of modular forms can be constructed explicitly in terms of a Hauptmodul; this is done in the proof of Theorem~2 in~\cite{BPf} and briefly recalled here. 

Assume that~$\Gamma$ has a cusp at~$i\infty$, let~$t$ a Hauptmodul as above, and let~$J:=\tfrac{1}{t}$ be a Hauptmodul with a simple pole at~$i\infty$. There exists a modular form~$\Delta=\Delta_{i\infty}$, of rational weight and possibly non-trivial multiplier system, such that
\begin{equation}
\label{eq:delta}
\mathrm{div}(\Delta)\=1\cdot i\infty\,,\quad \mathrm{weight}(\Delta)=\frac{-2}{\chi(\Gamma)}\,,
\end{equation}
and with~$p$-integral~$t$-expansion at the cusps for almost every prime~$p\not\in S$. 
If~$\Gamma$ is torsion-free, then it has~$s\ge3$ cusps~$c_1,\dots,c_{s-1},i\infty$, and~$\Delta$ can be defined as~$\Delta:=\bigl(J'/\prod_{l=1}^{s-1}(J(c_l)-J)\bigr)^{1/(s-2)}$, and every modular form is a positive power of~$\Delta$ times a polynomial in~$J$.
If~$\Gamma$ has~$r\ge1$ elliptic points~$e_1,\dots,e_r$, there exist classical modular forms~$Q_1,\dots,Q_r$, of rational weight and non-trivial multiplier system, with divisor~$\mathrm{div}(Q_i)=\frac{1}{n_i}\cdot e_i$, where~$n_i$ is the order of the elliptic point~$e_i$. They can be defined in terms of~$J$ and~$J'$ by slightly more complicated formulae than in the torsion-free case (see the proof of Theorem 2 in~\cite{BPf}). If~$\Gamma$ has a modular embedding, these modular forms can be chosen to have~$p$-integral~$t$-expansion at the cusps for every prime~$p\not\in S$, and be normalized to satisfy
\begin{equation}
\label{eq:Qi}
Q_i^{n_i}\=\frac{J-J(e_i)}{J-J(e_l)}\cdot Q_l^{n_l}\,,\quad i,l\in\{1,\dots,r\}\,.
\end{equation}
The ring of modular forms~$M_*(\Gamma)$ is generated by polynomials in~$Q_1,\dots,Q_r$, and we define the modular form~$\Delta$ by~$\Delta\cdot (J-J(e_1))=Q_1^{n_1}$.

\begin{lemma}
\label{lem:deco}
Let~$\Gamma$ be a genus-zero group with modular embedding~$\varphi$ and~$s\ge1$ cusps and~$r\ge0$ elliptic points~$e_1,\dots,e_r$ of order~$n_1,\dots,n_r$ respectively. Let ~$J$ be a Hauptmodul with a pole at~$i\infty$, and let~$\Delta$ and~$Q_1,\dots,Q_r$ be the modular forms introduced above, and let~$B_2,\dots,B_g$ be the twisted modular forms in the theorem. Let~$\vec{k}=(k_1,\dots,k_g)\in\Q\times\Z^{g-1}$ and~$v$ be a multiplier system. Then every twisted modular form~$f\in M_{\vec{k}}(\Gamma,\varphi,v)$ can be written uniquely as
\begin{equation}
\label{eq:deco}
f\=\Delta^{m}\cdot\prod_{i=1}^r{Q_i^{\epsilon_i}}\cdot \prod_{j=2}^g B_j^{k_j}\cdot\widetilde{f}(J)\,,
\end{equation}
where~$m\in\Z_{\ge0}$, $\epsilon_i\in\{0,1,\dots,n_i-1\}$, and~$\widetilde{f}(J) \in \C[J]$ is a polynomial of degree~$\le m$. 

Moreover, if the expansion at~$i\infty$ of~$f$ in the Hauptmodul~$t=1/J$ is~$p$-integral for a prime~$p$, then the polynomial~$\widetilde{f}(J)$ has~$p$-integral coefficients.
\end{lemma}

\begin{proof}
From Part 2 of Theorem~\ref{thm:Bj}, there exists a modular form~$f_1$ (of weight~$\sum_{j=1}^g{\lambda_jk_j}$) on~$\Gamma$ such that~$f=f_1\prod_{j=2}^gB_j^{k_j}$. Notice that~$\mathrm{div}(f)=\mathrm{div}(f_1)$ because of Part 1 of~Theorem~\ref{thm:Bj}. If~$\Gamma$ is torsion-free, it follows from the above discussion that~$f_1$ is a power of~$\Delta$ times a polynomial in~$J$, and we are done. Assume then that~$\Gamma$ has~$r\ge1$ elliptic points. Write
\[
\mathrm{div}(f_1)\=\sum_{i=1}^r{\frac{a_i}{n_i}\cdot e_i}\+\sum_{\tau\in\overline{\Po/\Gamma}\setminus\{e_1,\dots,e_r\}}{\mathrm{ord}_\tau{(f_1)}\cdot\tau}\,
\]
where~$a_i=b_i\cdot n_i+\epsilon_i$ uniquely for~$b_i\in\Z$ and~$\epsilon_i\in\{0,\dots,n_i-1\}$.
It follows that the divisor of~$f_1\prod_{i=1}^r{Q_i^{-\epsilon_i}}$ is effective, since~$f$ is holomorphic, and of integral degree~$m\in\Z_{\ge0}$. It follows that 
\[
f_1\prod_{i=1}^rQ_i^{-\epsilon_i}\=\prod_{i=1}^rQ_i^{b_in_i}\cdot f_2\,,
\]
where~$f_2$ is a modular form with zeros outside the elliptic locus, and therefore~$f_2=\Delta^{m_2}\widetilde{f}_2(J)$ is a power of~$\Delta$ times a polynomial in~$J$. From the relations~$\Delta\cdot(J(e_i)-J)=Q_1^{n_1}$ and~\eqref{eq:Qi} it follows that
\[
f_1\prod_{i=1}^rQ_i^{-\epsilon_i}\=\prod_{i=1}^rQ_i^{b_in_i}\cdot f_2\=\Delta^{\sum_{i=1}^rb_i+m_2}\cdot\prod_{i=1}^r{(J(e_i)-J)^{b_i}}\cdot\widetilde{f}_2(J)\,.
\]
The~$p$-integrality for almost every~$p\not\in S$ of the coefficients of~$\widetilde{f}$ follows from the~$p$-integrality of the~$t$-expansion at~$i\infty$ of all the twisted modular forms appearing in the product~\eqref{eq:deco}.
\end{proof}

\subsection{Partial Hasse invariants}
\label{sec:pHi}

Let~$p$ be a prime unramified in~$F$. In Section~\ref{sec:Koba} we mentioned that the non-ordinary locus in~$\mathcal{M}_{F,s}$ over~$s=\mathrm{Spec}(\overline{\F}_p)$ is described by divisors of certain Hilbert modular forms called partial Hasse invariants. The non-ordinary locus of the family~\eqref{eq:modelp} induced by a curve in~$M_f$ is then described by the restriction of the partial Hasse invariants to the base curve. The main result of~\cite{BLpHi} is that, contrary to the Hilbert case, (powers of) restrictions of partial Hasse invariants to genus zero curves lift to twisted modular forms in characteristic zero. 
The properties of the lifts that are of interest for this paper are recalled in the next theorem. 

\begin{theorem*}[Theorem 1 in~\cite{BLpHi}]
\label{thm:pHi}
Let~$Y\simeq\Po/\Gamma\hookrightarrow M_F$ have an integral model over~$R=\mathcal{O}_K[S^{-1}]$, and let~$p\not\in S$ be a prime.
Denote by~$N$ the least common multiple of the orders of the elliptic points of~$\Gamma$. There exist~$g$ twisted modular forms~$h_{p,1},\dots,h_{p,g}$ on~$\Gamma$ of trivial multiplier system, such that:
\begin{enumerate}
\item For every~$j=1,\dots,g$, the twisted modular form~$h_{p,j}(\tau)$ has~$p$-integral~$t$-expansion at the cusps. The weight of~$h_{p,j}$ is~$(0,\dots,0,-N,0,\dots,0,pN,0,\dots,0)$, where~$-N$ is in the~$j$-th entry, $pN$ in the~$j'$-entry. The value of~$j'\in\{1,\dots,g\}$ is uniquely determined by~$j$ and by the splitting behavior of~$p$ in~$F$. 
\item The twisted modular form~$h_{p,j}(\tau)$ has product decomposition
\[
h_{p,j}(\tau)=\Delta(\tau)^{m_{p,j}}B_{j}^{-N}B_{j'}^{pN}\cdot\widetilde{h}_{p,j}(J)\,,
\]
where~$m_{p,j}=\frac{\chi(Y)}{2}\cdot N\cdot (\lambda_j-p\lambda_{j'})$ with~$\lambda_j$ as in~\eqref{eq:Lyap}, and we set~$B_1=1$.
Moreover, the polynomial~$\widetilde{h}_{p,j}(J)$ is defined over~$R$ and, for every prime~$\fp\subset\mathcal{O}_K$ above~$p$, it satisfies 
\begin{equation}
\label{eq:gcd}
\frac{\widetilde{h}_{p,j}(J)}{\gcd\bigl(\widetilde{h}_{p,j}(J),\tfrac{d}{dJ}\widetilde{h}_{p,j}(J)\bigr)}\;\equiv\;\mathrm{ph}_{\fp,j}(J) \mod \fp\,,
\end{equation}
where~$\mathrm{ph}_{\fp,j}(J)$ is the~$j$-th partial Hasse polynomial defined in Section~\ref{sec:Koba}. 
\item The~$t$-expansion of~$h_{p,j}(\tau)$ at~$i\infty$ is constant and equal to~$1$ modulo~$p$.
\end{enumerate}
\end{theorem*}

The second property in the theorem is the analogue for the~$t$-expansion of the $q$-expansion property of the Hasse invariant for elliptic curves. It will be crucial in relating the Atkin polynomials to the non-ordinary locus of~$\mathcal{X}_{\overline{\fp}}$. 

\begin{example}
\label{ex:pHi2}
If~$g=2$ and~$p$ is a prime inert in~$F$, then~$h_{p,1}$ and~$h_{p,2}$ have weight~$(-1,p)$ and~$(p,-1)$ respectively, i.e., $j'=2$ if~$j=1$ and~$j'=1$ if~$j=2$. If~$p$ is split in~$F$, then~$h_{p,1}$ has weight~$(p-1,0)$ and~$h_{p,2}$ has weight~$(0,p-1)$, i.e., $j=j'$ for~$j=1,2$.
\end{example}

\begin{remark}
\label{rmk:ms}
The modular forms~$h_{p,j}$ in the theorem are lifts of powers of the partial Hasse invariants of~$\mathcal{X}_{\overline{\fp}}$. It is shown in~\cite{BLpHi} (where the modular forms in the theorem are denoted by~$h_{p,j}^N(\tau)$) that in many cases one can lift directly the partial Hasse invariants, which have lower weight (non-parallel weight~$p-1$) as well as reduced divisor. For the purposes of this paper, however, we only need the existence of twisted modular forms that satisfy the three properties of the theorem. One could (as we initially did) prove the theorems of this paper by using the minimal lift of the partial Hasse invariants, but with no real gain and at the expense of slightly longer proofs.  
\end{remark}

\subsection{Serre derivatives}
\label{sec:Serre}

Let~$\Gamma$ be a genus-zero non co-compact group with modular embedding~$\varphi$, and let~$\left(\begin{smallmatrix} 1 & \alpha_\infty \\ 0 & 1\end{smallmatrix}\right)$ be the generator of the free part of the stabilizer of the cusp~$i\infty$.
Normalize the differentiation operator~$D$ on~$\Po$ by 
\begin{equation}
\label{eq:D}
Df\=f'\:=\frac{\alpha_\infty}{2\pi i}\frac{df(\tau)}{d\tau}\=q\frac{df(q)}{dq}\,,
\end{equation}
for~$f$ a holomorphic function on~$\Po$ and~$q:=e^{2\pi i\tau/\alpha_\infty}$ a parameter at the cusp~$i\infty$. 

Let~$\Delta$ be the modular form defined before~\eqref{eq:delta}, and, for~$j=1,\dots,g$, let~$B_j$ be as in Theorem~\ref{thm:Bj}. Define
\begin{equation}
\label{eq:Pj}
P_1\:=\frac{\Delta'}{\Delta}\,,\quad P_j\:=\frac{\Delta'}{\Delta}\-\frac{2}{\lambda_j\cdot\chi(\Gamma)}\frac{B_j'}{B_j}\,,\quad j=2,\dots,g\,.
\end{equation}
Let~$\vec{k}=(k_1,\dots,k_g)\in\Q^g$. The~\emph{Serre derivative of weight~$\vec{k}$} is the differential operator
\begin{equation}
\label{eq:Serre}
\vartheta\=\vartheta_{\vec{k}}:= D\+\frac{\chi(\Gamma)}{2}\sum_{j=1}^g{k_j\lambda_jP_j}\,.
\end{equation}

\begin{lemma}
Let~$\vec{k}=(k_1,\dots,k_g)\in\Q^g$ and let~$f\in M_{\vec{k}}(\Gamma,\varphi)$. Then
\[
\vartheta_{\vec{k}}f\;\in\;M_{(k_1+2,k_2,\dots,k_g)}(\Gamma,\varphi)\,.
\]
\end{lemma}

\begin{proof}
The proof is a straightforward verification. Since~$\Delta$ and~$B_j$ are twisted modular forms of weight~$(-2/\chi(\Gamma),0,\dots,0)$ and~$(-\lambda_j,0,\dots,0,1,0,\dots,0)$, with~$1$ in the~$j$-th entry, respectively, for~$\gamma=\left(\begin{smallmatrix}a & b\\c&d \end{smallmatrix}\right)$ it holds
\[
\begin{aligned}
\frac{\Delta'(\gamma\tau)}{\Delta(\gamma\tau)}&\=(c\tau+d)^2\biggl(\frac{\Delta'(\tau)}{\Delta(\tau)}\+\frac{-\alpha_\infty}{\pi i\chi(\Gamma)}\cdot\frac{c}{c\tau+d}\biggr)\,,\\
\frac{B_j'(\gamma\tau)}{B_j(\gamma\tau)}&\=(c\tau+d)^2\biggl(\frac{B_j'(\tau)}{B_j(\tau)}\+\frac{\alpha_\infty}{2\pi i}\cdot\frac{\sigma_j(c)\varphi_j'(\tau)}{\sigma_j(c)\varphi_j(\tau)+\sigma_j(d)}\-\frac{\alpha_\infty\lambda_j}{2\pi i}\cdot\frac{c}{c\tau+d}\biggr)\,.
\end{aligned}
\] 
It follows that
\begin{equation}
\label{eq:trPj}
P_j(\gamma\tau)\=(c\tau+d)^2\biggl(P_j(\tau)+\frac{-\alpha_\infty}{\pi i\cdot\lambda_j\chi(\Gamma)}\cdot\frac{\sigma_j(c)\varphi_j'(\tau)}{\sigma_j(c)\varphi_j(\tau)+\sigma_j(d)}\biggr)\,.
\end{equation}
In other words, the~$P_j$ are analogues of quasimodular forms of weight two (e.g., the classical weight-two Eisenstein series) in the twisted setting. If~$f\in M_{\vec{k}}(\Gamma, \varphi)$, then
\[
f'(\gamma\tau)\=(c\tau+d)^2\prod_{j=1}^g(\sigma_j(c)\varphi_j(\tau)+\sigma_j(d))^{k_j}\cdot\biggl(f'(\tau)\+\frac{\alpha_\infty f(\tau)}{2\pi i}\cdot\sum_{j=1}^g\frac{k_j\sigma_j(c)\varphi_j'(\tau)}{\sigma_j(c)\varphi_j(\tau)+\sigma_j(d)}\biggr)\,,
\]
and a simple computation using the definition of~$\vartheta_{\vec{k}}$ proves the lemma. 
\end{proof}

\section{Non-ordinary locus and zeros of orthogonal polynomials}
\label{sec:main}

\subsection{Orthogonal polynomials of one variable}
\label{sec:ortho}

Here are collected the needed facts on orthogonal polynomials. We refer to Section~4 of~\cite{KZ}, from which we borrow the exposition, for a more detailed discussion and the proof of the results of this section. 

Let~$K$ be a field.
We will consider a scalar product~$\langle\,,\rangle$ on~$V=K[x]$ of the form~$\langle g,h\rangle=\psi(gh)$ where~$\psi\colon V\to K$ is a linear functional (in particular, $\langle g,h\rangle$ only depends on the product~$gh$). Applying the Gram–Schmidt process to the basis~$\{x^n\}_{n\ge0}$ of~$V$ we obtain a unique basis of monic orthogonal polynomials~$P_n(x)$ via the recursive definition
\begin{equation}
\label{eq:GSp}
P_n(x)\=x^n-\sum_{m=0}^{n-1}{\frac{\langle x^n,P_m\rangle}{\langle P_m,P_m\rangle }P_m(x)}\,,
\end{equation}
provided that at each stage~$\langle P_m,P_m\rangle\ne0$. The second part of the next proposition will play a major role in relating orthogonal polynomials to the non-ordinary locus of curves in Hilbert modular varieties.

\begin{proposition*}[Proposition~2 in~\cite{KZ}]
\label{prop:ortho}
\begin{enumerate}[wide=0pt]
\item The polynomials~$P_n$ satisfy a three-term recursion of the form
\[
P_{n+1}(x)=(x-a_n)P_n(x)-b_nP_{n-1}(x)\quad n\ge1\,,
\]
for some constants~$a_n,b_n\in K$ and~$b_n=\frac{\langle P_n,P_n\rangle}{\langle P_{n-1},P_{n-1}\rangle}\ne0$.
\item Define a second sequence of polynomials~$\{Q_n\}_{n\ge0}$ in~$K[x]$ by the same recurrence as in~1., but with initial values~$Q_0=0$ and~$Q_1(x)=\psi(1)$. Then
\[
\frac{Q_n(x)}{P_n(x)}\=\Phi(x)+O(x^{-2n-1})\in K[[x^{-1}]]\,,
\]
where
\[
\Phi(x)=\sum_{n=0}^\infty{\langle x^n,1\rangle x^{-n-1}}\in K[[x^{-1}]]
\]
is the generating function of moments of the scalar product~$\langle\,,\rangle$. 
This property characterizes~$P_n$ (assumed to be monic and of degree~$n$) and~$Q_n$ uniquely. 
\item Let~$g_n:=\langle x^n,1\rangle$ and define numbers~$\lambda_n\in K$ for~$n\ge1$ by the continued fraction expansion
\[
g_0+g_1X+g_2X^2+\cdots\=\frac{g_0}{1-\frac{\lambda_1X}{1-\frac{\lambda_2X}{1-\ddots}}}\;\in\;K[\![X]\!]\,.
\]
Then all~$\lambda_n$ are non-zero and~$a_n=\lambda_{2n}+\lambda_{2n+1}$ and~$b_n=\lambda_{2n-1}\lambda_{2n}$ for~$n\ge1$.
\end{enumerate}
\end{proposition*}

\subsection{Atkin's polynomials}

Let~$Y\hookrightarrow M_F$ be a genus-zero affine curve in a Hilbert modular variety of dimension~$g$, with integral model~$\mathcal{X}\to\mathcal{Y}$ defined over~$R=\mathcal{O}_K[S^{-1}]$. Let~$\Gamma$ be the uniformizing group of~$Y$, and denote by~$\varphi=(\varphi_j)_{j=1}^g\colon\Po\to\Po^g$ the associated modular embedding. Let~$t$ be the Hauptmodul with zero at the cusp~$i\infty$ introduced in Section~\ref{sec:tmf0}, and let~$J:=\tfrac{1}{t}$ be a Hauptmodul with a pole at~$i\infty$. Let~$P_1,\dots,P_g$ be the 
twisted quasimodular forms defined in~\eqref{eq:Pj}. 

For~$j=1,\dots,g$ we define scalar products~$\langle\,,\rangle_j$ on the vector space~$K[J]$ by
\[
\langle g,h\rangle_j\:= \text{ the constant term in the Laurent $t$-expansion of }-g\cdot h\cdot\frac{JP_j}{J'}\,,\quad g,h\in K[J]\,.
\]
We call~$\langle\,,\rangle_j$ the~\emph{$j$-th Atkin scalar product}.

\begin{example}
  \label{ex:KZ}
For~$\Gamma=\SL_2(\Z)$, then~$g=1$ and~$K=\Q$, and the definition of~$\langle\,,\rangle_1$ agrees with the definition of Atkin's scalar product in Kaneko–Zagier's paper~\cite{KZ} up to a non-zero constant factor. In this case we can normalize~$Q_1=E_4$ to be the Eisenstein series of weight~$4$, and let~$J$ be the usual~$J$-invariant function. From the relation after~\eqref{eq:Qi} we find that our~$\Delta=E_4^3/J$ is the usual normalized cusp form of weight~$12$. Then~$P_1=\tfrac{\Delta'}{12\Delta}$ in~\eqref{eq:Pj} satisfies~$P_1=12\cdot E_2$ and, since~$J/J'=E_4/E_6$, we have that~$\langle g,h\rangle_1=gh\cdot 12 E_2E_4/E_6=12\cdot\langle g,h\rangle_A$, where~$\langle\,,\,\rangle_A$ is the Atkin scalar product in~\cite{KZ}, by Point 3 of Proposition 3 in that paper. 
\end{example}

\begin{lemma}
\label{lem:sprod}
\begin{enumerate}[wide=0pt]
\item Let~$q$ be as in Section~\ref{sec:Serre}. For~$j=1,\dots,g$, the following definitions of the scalar products~$\langle\,,\rangle_j$ are equivalent:
\begin{enumerate}[wide=0pt]
\item $\langle g,h\rangle_j\:=$ the constant term in the Laurent~$t$-expansion of $-gh\frac{JP_j}{J'}$;
\item $\langle g,h\rangle_j\:=$ the constant term in the Laurent~$q$-expansion of $gh\cdot P_j$.
\end{enumerate}
\item The generating function~$\Phi_j(t)=\sum_{m\ge0}\langle J^m,1\rangle_j\cdot t^{m+1}$ of moments of~$\langle\,,\,\rangle_j$ is given in closed form by
\begin{equation}
\label{eq:genfun}
\Phi_j(t)\=-\frac{P_j(t)}{J'(t)}\,.
\end{equation}
\end{enumerate}
\end{lemma}

\begin{proof}
\begin{enumerate}[wide=0pt]
\item
The equivalence between the definitions follows by writing the constant term of the Laurent expansion as an integral with Cauchy's formula
\[
\langle g,h\rangle_j\=\frac{1}{2\pi i}\int_{\delta_\infty}\frac{-g(t)h(t)\cdot \tfrac{J(t)}{J'(t)}P_j(t)}{t}\,dt\,,
\]
for~$\delta_\infty$ a loop around~$i\infty$, and by using the relation~$-P_j(t)\tfrac{J(t)}{J'(t)}\,\frac{dt}{t}=P_j(t)\frac{dt}{t'}=P_j(q)\frac{dq}{q}$. 
\item From the definition, the Atkin scalar product~$\langle J^m,1\rangle_j$ is the constant term in the expansion of~$-J^m\cdot1\cdot\tfrac{JP_j}{J'}=-J^{m+1}\frac{P_j}{J'}$ in~$t=J^{-1}$, i.e., the~$m$-th coefficient in the expansion~$-\frac{P_j}{J'}=\sum_{n=0}^\infty{g_n\cdot J^{-n-1}}=\sum_{n=0}^\infty{g_nt^{n+1}}$ (the power series expansion of~$P_j/J'$ starts with~$t$ since~$J'$ has a simple pole at~$i\infty$).
It follows that the generating function~$\Phi_j(t)$ of moments~$\langle J^m,1\rangle_j$ is expressed in closed form as in the statement. 
\end{enumerate}
\end{proof}

\begin{remark}
\label{rmk:deltaj}
In Kaneko–Zagier's paper, the Atkin scalar product~$\langle g,h\rangle_A$ for modular functions on~$\SL_2(\Z)$ is originally defined as the constant term in the Laurent expansion of~$g\cdot h$ in the parameter~$\Delta$ at~$i\infty$. This is equivalent to our definition of~$\langle\,,\rangle_1$ via the relation~$P_1=\Delta'/\Delta$ and an argument analogous to the one of Lemma~\ref{lem:sprod}.
Similarly to~$\Delta$ one can define for~$j=2,\dots,g$ a twisted modular form~$\Delta_j$ of weight~$\frac{-2}{\chi(\Gamma)\lambda_j}$ in the~$j$-th entry and zero elsewhere, and with degree-one divisor supported at~$i\infty$. There are two ways to construct such~$\Delta_j$: one is to replace~$t'$ with~$t'/\varphi_j'$ in the explicit formulae for~$Q_i$ in Theorem~2 of~\cite{BPf}, in order to get twisted modular forms of minimal divisor analogous to~$Q_1,\dots,Q_r$, and then proceed as in Section~2. Another possibility, using the twisted modular forms~$B_j$, is to choose a branch of~$B_j^{r_j}$, where~$r_j=\frac{-2}{\chi(\Gamma)\lambda_j}$, and define~$\Delta_j=\Delta\cdot B_j^{r_j}$. 

By considering the twisted modular forms~$\Delta_j$, one can give a definition of the scalar product~$\langle\,,\rangle_j$ analogous to the one in~\cite{KZ}, which is equivalent to the one we gave here. The twisted modular forms~$\Delta_j$ can be used to describe the functions~$P_j$ of Section~\ref{sec:Serre} as
\[
P_j\=\frac{\Delta'_j}{\Delta_j}\=\frac{\alpha_\infty\cdot d\log(\Delta_j)}{2\pi i\cdot d\tau}\,,\quad j=1,\dots,g\,,
\]
and Part (b) of Lemma~\ref{lem:sprod} can be restated as 
\begin{equation}
\label{eq:Pjv}
\Phi_j(t)\=-\frac{P_j(t)}{J'(t)}\=\frac{t^2}{t'}\cdot \frac{\alpha_\infty\cdot d\log(\Delta_j(\tau))}{2\pi i\cdot d\tau}\=t^2\cdot\frac{d\log(\Delta(t))}{dt}\,,
\end{equation}
since~$J'=\frac{-t'}{t^2}$.
\end{remark}

\begin{definition}
  \label{def:Ajn}
Let~$j\in\{1,\dots,g\}$, and suppose that the~$j$-th Atkin scalar product~$\langle\,,\rangle_j$ satisfies $\langle x,x\rangle_j\ne0$ for every non-zero~$x\in K[J]$. We define the~\emph{$j$-th Atkin's polynomials}~$\{A_{j,n}(J)\}_n$ as the orthogonal polynomials obtained by the Gram–Schmidt process applied to the basis~$\{J^n\}_{n\ge0}$ of~$K[J]$ with respect to the scalar product~$\langle\,,\rangle_j$.
\end{definition}

\subsection{Positive definiteness of Atkin scalar products}
\label{sec:posdef}
In the definition of Atkin polynomials we supposed that the scalar products satisfy the condition~$\langle x,x\rangle_j\ne0$ for every non-zero~$x\in K[J]$. However, as shown in~\cite{Ts} for the product~$\langle\,,\rangle_1$ on~$\Gamma_0(5)$, this is not always the case. Nevertheless, the Atkin scalar products satisfy this condition for infinitely many Fuchsian groups $\Gamma$.

\begin{proposition}
\label{prop:posdef}
Let~$Y\simeq\Po/\Gamma$ be a genus zero affine curve with a modular embedding defined over~$K\subset \R$. Then there exists a Hauptmodul~$J$ such that the scalar products~$\langle\,,\rangle_j$ are positive definite on~$\R[J]$ for every~$j=1,\dots,g$. 
\end{proposition}

It follows that the Atkin polynomials exist for all triangle groups of type~$\Delta(n,m,\infty)$ and~$m,n\ge2$, whose associated quotient space is isomorphic to~$\pro^1\smallsetminus\{\infty,1,0\}$. Other examples include genus-zero Weierstrass–Teichmüller curves~$W_D$ in Hilbert modular surfaces, which are defined over~$\Q(\sqrt{D})$, for~$D>0$ a positive discriminant (see~\cite{BM} for examples with~$D=13$ and~$D=17$).
Finally, sometimes one can find a genus-zero branched covering defined over~$\R$ of a punctured sphere with complex algebraic punctures; an example, considered in~\cite{Sa}, is that of the curve~$\Po/\Gamma_0(5)$, whose covering curve~$\Po/\Gamma_0^*(5)$ is isomorphic to a real punctured sphere.

\begin{proof}
We first construct a suitable Hauptmodul~$J$, and then use it to prove the positive definiteness of the Atkin scalar products. 
In order to do this, we recall basic facts from the theory of conformal mappings and uniformization. A reference for the role of automorphic functions in classical uniformization is Chapter IX of Ford's book~\cite{Ford}. 

A punctured sphere~$Y=\pro^1\setminus\{a_1,\dots,a_n,\infty\}$ with~$a_1<a_2<\cdots<a_n\in\R$ is invariant under complex conjugation. It follows that if~$Y\simeq\Po/\Gamma$, then the Fuchsian group~$\Gamma$ admits a symmetric fundamental domain~$\mathcal{F}=\mathcal{F}_1\cup\mathcal{F}^*_1$, where~$\mathcal{F}_1^*$ is the reflection of~$\mathcal{F}_1$ along one of its boundary geodesics (see~\cite{Sibner} for a general statement). For such Riemann surfaces, the uniformization problem, that is, the construction of a holomorphic universal covering map~$J\colon\Po\to Y$ (a Hauptmodul), is equivalent to the problem of finding a conformal map~$\eta\colon\{y\in Y : \mathrm{Im}(y)>0\}\subset Y\to\mathcal{F}_1$; the universal covering map is obtained in fact by setting~$J=\eta^{-1}$ on~$\mathcal{F}_1$ and~$J=\overline{\eta^{-1}}$ (complex conjugation) on~$\mathcal{F}_1^*$, and by extending it to~$\Po$ via Schwarz reflection.
 The conformal map~$\eta$ always exists, and extends to the boundary: in particular, it maps the real segments connecting consecutive punctures~$a_i$ and~$a_{i+1}$ into suitable boundary geodesics of~$\mathcal{F}_1$ (we are assuming that~$J(i\infty)=\infty$). It follows that the Hauptmodul~$J$ maps elliptic points and cusps (the vertices of~$\mathcal{F}_1$) into the punctures~$a_1,\dots,a_n$ and takes real values on the boundary of~$\mathcal{F}$. In particular, this construction applies to~$Y$ with modular embedding defined over a number field~$K\subset\R$, yielding a Hauptmodul~$J$ taking algebraic values at elliptic points and cusps, and real values on the boundary of~$\mathcal{F}$.

Denote by~$\mathcal{F}_y$ the polygon obtained by truncating~$\mathcal{F}$ at height~$y>>0$, let~$\delta_\infty$ be its top edge, and let~$\delta_0$ be the image of~$\delta_\infty$ via the map~$\tau\mapsto\exp(2\pi i\tau/\alpha_\infty)$, where~$\alpha_\infty$ is as in~\eqref{eq:D} (it is the length of~$\delta_\infty$).  

From the definition of the Atkin scalar products and Lemma~\ref{lem:sprod}, we can write, for~$g,h\in K[J]$ and~$j=1,\dots,g$,
\begin{equation}
\label{eq:at}
\langle g,h\rangle_j\=\frac{1}{2\pi i}\int_{\delta_0}{g(J(q))h(J(q))\cdot P_j(q)\,\frac{dq}{q}}\=\frac{1}{\alpha_\infty}\int_{\delta_\infty}{g(J(\tau))h(J(\tau))\cdot P_j(\tau)\,d\tau}\,,
\end{equation}
where we see a polynomial in~$J$ as a modular function, i.e., as a function of~$q$ or~$\tau$.
Since the integrand is holomorphic on~$\mathcal{F}_y$, we can express the line integral as the integral over the part of the boundary of~$\mathcal{F}_y$ complementary to~$\delta_\infty$. 
Because of the symmetry, the sides of~$\mathcal{F}_y$ come in pairs~$(\delta_s,\delta'_s)$ together with a matrix~$M_s\in\Gamma$ such that~$\delta_s'=-M_s(\delta_s)$ (reversed orientation). By the invariance under~$\Gamma$ of~$g$ and~$h$ and Equation~\eqref{eq:at} it follows 
\[
\begin{aligned}
-\alpha_\infty\langle g,h\rangle_j&\=\sum_{s}\int_{\delta_s}{(gh)(J(\tau))\cdot P_j(\tau)\,d\tau}\+\int_{M_s(\delta_s)}{{(gh)(J(\tau))\cdot P_j(\tau)\,d\tau}}\\
&\=\sum_{s}\int_{\delta_s}{(gh)(J(\tau))\bigl[P_j(\tau)- P_j(M_s\tau)(c_s\tau+d_s)^{-2}\bigr]\,d\tau}\,,\quad M_s=\begin{pmatrix} a_s & b_s \\c_s & d_s
\end{pmatrix}\,. 
\end{aligned}
\]
The modular behavior~\eqref{eq:trPj} of~$P_j$ applied to the transformation~$M_s$ implies that
\begin{equation}
\label{eq:step1}
\alpha_\infty\cdot\langle g,h\rangle_j=\frac{\alpha_\infty}{2\pi i}\cdot\sum_{s}\int_{\delta_s}{(gh)(J(\tau))\cdot\frac{\varphi_j'(\tau)}{\varphi_j(\tau)+\sigma_j(d_s/c_s)}d\tau}\,.
\end{equation}

Let~$J_j$ be the holomorphic map defined on~$\varphi_j(\mathcal{F})$ such that~$J_j\circ\varphi_j(\tau)=J(\tau)$ for~$\tau\in\mathcal{F}$. These maps, whose existence is guaranteed by Riemann's mapping theorem, can be computed from the solutions of certain Picard–Fuchs differential equation associated to~$Y$ (see~\cite{BPf}). 
Note that, from the realness of~$J$ on the boundary of~$\mathcal{F}$ it follows that also~$J_j$ takes real values over~$\varphi_j(\delta_s)$, for every~$j=2,\dots,g$.

By rewriting a polynomial~$f(J)\in K[J]$ as~$f(J(\tau))=f(J_j(\varphi_j(\tau)))$ and by changing the variable from~$\tau$ to~$\varphi_j$ in~\eqref{eq:step1} one gets
\[
2\pi\cdot\langle g,h\rangle_j\=\frac{1}{i}\sum_s\int_{\varphi_j(\delta_s)}{\frac{(gh)(J_j(\varphi_j))}{\varphi_j+\sigma_j(d_s/c_s)}\,d\varphi_j}\,.
\]
The boundary geodesic~$\delta_s$ and its image via~$\varphi_j$ can be written as
\[
\delta_s\=-\frac{d_s}{c_s}+r_{s,1}\cdot e^{i\theta_1}\,,\qquad\varphi_j(\delta_s)=-\frac{\sigma_j(d_s)}{\sigma_j(c_s)}+r_{s,j}\cdot e^{i\theta_j}\,,
\]
for some positive radii~$r_{s,j}$ and angles~$\theta_j\in[a_{s,j},b_{s,j}]\subset[0,\pi]$.
By expressing the integral with respect to the variable~$\theta_j$ one finally gets
\begin{equation}
\label{eq:posdef2}
2\pi\cdot\langle g,h\rangle_j\=\sum_s\int_{a_{s,j}}^{b_{s,j}}{(gh)\bigl(J_j\bigl(-\sigma_j(d_s/c_s)+r_{s,j}\cdot e^{i\theta_j}\bigr)\bigr)\,d\theta_j}\,.
\end{equation} 
From~\eqref{eq:posdef2} one computes that
\[
2\pi\cdot\langle g,g\rangle_j\=\sum_s\int_{a_{s,j}}^{b_{s,j}}{g\bigl(J_j\bigl(-\sigma_j(d_s/c_s)+r_{s,j}\cdot e^{i\theta_j}\bigr)\bigr)^2\,d\theta_j}\,.
\]
This, together with the realness of~$J_j$ on~$\varphi_j(\delta_s)$ for every~$s$, proves that the Atkin scalar products are positive definite on~$\R[J]$. 
\end{proof}

\subsection{Partial Hasse polynomials and Atkin's polynomials}
\label{sec:ssl}

Let~$Y\simeq\Po/\Gamma$ be a curve in a Hilbert modular variety with model over~$\mathcal{O}_K[S^{-1}]$. 
From now on, we suppose that the set~$S$ contains all the primes dividing the numerator and the denominator of the orbifold Euler characteristic~$\chi(\Gamma)$ as well as the primes dividing numerator and denominator of the Lyapunov exponents~$\lambda_1,\dots,\lambda_g$ defined in the Theorem of Section~\ref{sec:tmf0}.

\begin{theorem}
\label{thm:main}
Let~$Y\hookrightarrow M_F$ be a non-compact genus-zero curve in a Hilbert modular variety of dimension~$g$ with integral model over~$R=\mathcal{O}_K[S^{-1}]$. For~$j\in\{1,\dots,g\}$, let $\{A_{j,n}(J)\}_n \subset K[J]$ be the family of Atkin orthogonal polynomials in Definition \ref{def:Ajn}. 
For every rational prime~$p\not\in S$ there exists~$n_{p,j}\in\Z_{\ge0}$ such that~$A_{j,n_{p,j}}(J)$ has~$p$-integral coefficients and
\[
A_{j,n_{p,j}}(J)\equiv\mathrm{ph}_{p,j}(J)\mod p\,,
\]
i.e., for every prime~$\fp\subset\mathcal{O}_K$ above~$p$ it holds
\[
A_{j,n_{p,j}}(J)\equiv\mathrm{ph}_{\fp,j}(J)\mod \fp\,.
\]
\end{theorem}

\begin{proof}
Fix~$j\in\{1,\dots,g\}$ and let~$p$ be a prime~$p\not\in S$ unramified in~$F$ and in~$K$. The following congruences hold for the generating function of the moments of~$\langle\,,\rangle_j$:
\begin{equation}
\label{eq:gfcong}
\Phi_j(t)\=-\frac{P_j(t)}{J'(t)}\;\equiv\;\frac{-P_j(t)}{J'(t)\cdot h_{p,j}(t)}\;\equiv\;\frac{2}{N\lambda_j\chi(\Gamma)}\frac{\vartheta_{\vec{k}}h_{p,j}(t)}{J'(t)\cdot h_{p,j}(t)}\mod p\,,
\end{equation}
where~$\vec{k}$ is the weight of the partial Hasse invariant~$h_{p,j}$, i.e., it is~$-N$ in the~$j$-th component, $pN$ in the~$j'$-th component, and zero elsewhere (the indices~$j$ and~$j'$ are related as in Theorem~\ref{thm:pHi}). 
The first congruence holds because~$h_{p,j}(t)\equiv1\mod p$ (Point 2 in Theorem~\ref{thm:pHi}), which implies that to multiply or divide by~$h_{p,j}$ does not change anything modulo~$p$. The second congruence in~\eqref{eq:gfcong} follows from the computation of the Serre derivative
\begin{equation}
\label{eq:Serrecong}
\begin{aligned}
\vartheta_{\vec{k}}h_{p,j}&\=Dh_{p,j}\+\frac{\chi(\Gamma)}{2}(-N\lambda_jP_j+pN\lambda_{j'} P_{j'})\cdot h_{p,j}\\
&\;\equiv\; \frac{1}{t'(t)}\frac{d\,1}{dt}\+\frac{\chi(\Gamma)}{2}(-N\lambda_jP_j+pN\lambda_{j'}P_{j'})\cdot 1\;\equiv\; 0-\frac{\chi(\Gamma)N\lambda_j}{2}P_j(t)\mod p\,,
\end{aligned}
\end{equation}
where in the first congruence we used again that~$h_{p,j}(t)\equiv 1\mod p$, and~$D=\frac{1}{t'}\frac{d}{dt}$.

Let~$h_{p,j}=\Delta^{m_{p,j}}B_j^{-N}B_{j'}^{pN}\cdot\widetilde{h}_{p,j}(J)$ be the product decomposition of the partial Hasse invariant~$h_{p,j}$. In order to compute its Serre derivative, notice first that
\[
\vartheta\Delta=D\Delta\+\frac{\chi(\Gamma)}{2}\frac{-2}{\chi(\Gamma)}P_1\cdot\Delta\=D\Delta\-P_1\cdot\Delta\=0
\]
by definition of~$P_1$. A similar computation shows that~$\vartheta B_l=0$ for every~$l=1,\dots,g$. Finally, from the computation~$\vartheta(\widetilde{h}_{p,j}(J))=D\widetilde{h}_{p,j}(J)=J'\cdot\frac{d\widetilde{h}_{p,j}(J)}{dJ}$, one concludes that
\[
\vartheta(h_{p,j})\=\Delta^{m_{p,j}}B_j^{-N}B_{j'}^{pN} J'\frac{d\widetilde{h}_{p,j}(J)}{dJ}\,.
\]
This identity, together with the congruence in Part 2 of the theorem in Section~\ref{thm:pHi}, implies that, for every prime~$\fp\subset \mathcal{O}_K$ above~$p$, it holds
\begin{equation}
\label{eq:step2}
\frac{2}{N\lambda_j\chi(\Gamma)}\frac{\vartheta (h_{p,j})}{J'\cdot h_{p,j}}\=\frac{2}{N\lambda_j\chi(\Gamma)}\frac{\frac{d\widetilde{h}_{p,j}(J)}{dJ}}{\widetilde{h}_{p,j}(J)}\;\equiv\;\frac{R_{p,j}(J)}{\mathrm{ph}_{\fp,j}(J)}\mod \fp\,,
\end{equation}
where
\[
R_{p,j}(J):=\frac{2}{N\lambda_j\chi(\Gamma)}\frac{\frac{d\widetilde{h}_{p,j}(J)}{dJ}}{\gcd\bigl(\widetilde{h}_{p,j}(J),\frac{d}{dJ}\widetilde{h}_{p,j}(J)\bigr)}\,.
\] 
Note that by definition of~$R_{p,j}(J)$, its modulo~$\fp$ reduction and~$\mathrm{ph}_{\fp,j}(J)$ are coprime over~$k_\fp[J]$.

Let~$n_{p,j}\in\Z_{\ge0}$ be the degree of~$\mathrm{ph}_{\fp,j}(J)$ (it depends only on~$p$, see Theorem 3 in~\cite{BPf}), and let~$A_{j,n_{p,j}}(J)$ be the polynomial of degree~$n_{p,j}$ in the family of~$j$-th Atkin's polynomials~$\{A_{j,n}(J)\}_n$.
The second part of Proposition~\ref{prop:ortho} and Equations~\eqref{eq:gfcong} and~\eqref{eq:step2} relate, via the generating function of moments~$\Phi_j(J^{-1})$, the partial Hasse polynomials with the Atkin polynomials modulo~$p$:
\begin{equation}
\label{eq:gfcong2}
\frac{R_{p,j}(J)}{\mathrm{ph}_{\fp,j}(J)}\;\equiv\;\Phi_j\bigl(J^{-1}\bigr)\=\frac{B_{n_{p,j}}(J)}{A_{j,n_{p,j}}(J)}+O(J^{-2n_{p,j}-1})\mod \fp\,,
\end{equation}
where~$B_{n_{p,j}}(J)$ is the polynomial of degree~$n_{p,j}-1$ obtained as in Proposition~\ref{prop:ortho}. 
The proof can be concluded similarly to~Kaneko–Zagier.  Congruence~\eqref{eq:gfcong2}, after multiplying~$A_{j,n_{p,j}}(J)$ and~$B_{n_{p,j}}(J)$ if necessary by a common power of a uniformizer of~$\fp$ in order to make~$A_{j,n_{p,j}}(J)$ a $\fp$-integral polynomial, implies that
\[
B_{n_{p,j}}(J)\cdot\mathrm{ph}_{\fp,j}(J)\-R_{p,j}(J)\cdot A_{j,n_{p,j}}(J)\;\equiv\; O(J^{-1})\mod \fp
\]
as~$J\to\infty$ and, since the expression is a polynomial, it must vanish. 
The coprimality of~$R_{p,j}(J)$ modulo~$\fp$ and~$\mathrm{ph}_{\fp,j}(J)$ implies that~$\mathrm{ph}_{\fp,j}(J)$ divides the modulo~$\fp$ reduction of~$A_{j,n_{p,j}}(J)$ and hence, since the degree of~${A_{j,n_{p,j}}(J)}\mod\fp$ is bounded by~$\deg(\mathrm{ph}_{\fp,j}(J))$, this implies that~$A_{j,n_{p,j}}(J)$ has~$\fp$-integral coefficients and reduces to~$\mathrm{ph}_{\fp,j}(J)$ modulo~$\fp$. Since this holds for every prime~$\fp$ dividing~$p$, the theorem follows. 
\end{proof}

\begin{remark}
  As the proof of Theorem~\ref{thm:main} reveals, the integer~$n_{p,j}$ in the statement of the theorem is the degree of the partial Hasse polynomial~$\mathrm{ph}_{p,j}(J)$. It has been proven in~\cite{BLpHi} in infinitely many cases (including curves $\Po/\Gamma$ where $\Gamma$ has at most one elliptic point, or $\Gamma=\Delta(n,m,\infty)$ with $n$ and~$m$ coprime)
 that
\[
n_{p,j}=\dim M_{\vec{k}_j}(\Gamma,\varphi)\-1\,,
\]
where the weight~$\vec{k}_j$ is~$-1$ in the~$j$-th entry, $p$ in the~$j'$-entry, and zero elsewhere, where~$j$ and~$j'$ are related as in Theorem~\ref{thm:pHi}. Since the dimension of the space of twisted modular form is easily computable, it makes Theorem~\ref{thm:main} and Corollary~\ref{cor:main} below suitable for the computation of the zeroes of partial Hasse invariants and of the non-ordinary locus. 
Examples of the calculation of~$n_{p,j}$ can be found in Section~2.4 of~\cite{BLpHi}. 
\end{remark}

\begin{corollary}
\label{cor:main}
Let~$Y\hookrightarrow M_F$ be as in Theorem~\ref{thm:main} with integral model over~$\mathcal{O}_K[S^{-1}]$. For~$p\not \in S$ let~$\mathrm{no}_p(J)$ and~$\mathrm{sp}_p(J)$ denote the non-ordinary and the superspecial polynomial of~$Y$ defined in Section~\ref{sec:Koba}. For every prime~$p\not\in S$ there exist~$n_{p,1},\dots,n_{p,g}\in\Z_{\ge0}$ such that
\[
\begin{aligned}
\mathrm{no}^\mathcal{Y}_p(J)&\;\equiv\;\mathrm{lcm}\bigl(A_{1,n_{p,1}}(J),\dots,A_{g,n_{p,g}}(J)\bigr)\mod p\,,\\
\mathrm{sp}^\mathcal{Y}_p(J)&\;\equiv\;\mathrm{gcd}\bigl(A_{1,n_{p,1}}(J),\dots,A_{g,n_{p,g}}(J)\bigr)\mod p\,.
\end{aligned}
\]
If~$g=2$, the following congruences hold for the supersingular polynomial~$\mathrm{ss}_p(J)$:
\[
\mathrm{ss}_p(J)\;\equiv\;
\begin{cases}
\mathrm{lcm}(A_{1,n_{p,1}}(J), A_{2,n_{p,2}}(J)) \mod p &\text{ if $p$ is inert in }F\,,\\
\mathrm{gcd}(A_{1,n_{p,1}}(J), A_{2,n_{p,2}}(J))\mod p &\text{ if $p$ is split in }F\,.
\end{cases}
\]
\end{corollary}

\begin{proof}
The statement follows from Theorem~\ref{thm:main} and the definition of the partial Hasse polynomials. For~$g=2$, the non-ordinary locus is the supersingular locus if~$p$ is inert in~$F$, and it is the superspecial locus (the intersection of the components of the non-ordinary locus) if~$p$ is split, as proven by Bachmat and Goren~\cite{BG}.
\end{proof}

\subsection{Hasse invariant and Atkin's polynomials}
\label{sec:Hasse}

We showed how to relate the non-ordinary locus of a genus-zero non-compact curve in a $g$-dimensional Hilbert modular variety to the zeros of~$g$ families of orthogonal polynomials. The goal of this section is to show that in principle only one family of orthogonal polynomials is needed. 

On the modular side, instead of the Hilbert partial Hasse invariants~$h_{j}$ on~$\mathcal{M}_{F,s}$ we consider the~\emph{Hasse invariant}~$h$ introduced in~Section~\ref{sec:Koba}. As mentioned there, the modular form $h$ (or a power of~$h$) admits a lift to a Hilbert modular form~$H_p$ in characteristic zero of parallel weight~$(p-1)$.
There are two ways to construct lifts of the Hasse invariant for a curve~$Y\hookrightarrow M_F$. The first one is by restricting the Hilbert modular form~$H_p$ to the curve~$Y$, i.e., by defining~$h_p(\tau)=h^Y_p(\tau):=H_p(\tau,\varphi_1(\tau),\dots,\varphi_g(\tau))$, where~$\varphi_1,\dots,\varphi_g$ denote the components of the modular embedding of~$\Gamma$. We remark that, even in the case where~$H_p$ is a Hilbert Eisenstein series, the restriction to~$Y$ is not in general an Eisenstein series on~$\Gamma$ (see~\cite{BLSpan} for a detailed study of this phenomenon for~$\Gamma=\SL_2(\Z)$). 

The second method consists in taking the product of the lifts of the partial Hasse invariants~$h_{j,p}$. In any case, the lift of the Hasse invariant~$h_p$ has the following properties: its reduction modulo~$\fp$ has zeros only at the non-ordinary locus of~$\mathcal{X}_{\overline{\fp}}$, for every prime~$\fp\subset\mathcal{O}_K$ over~$p$, and its~$t$-expansion at every cusp is constant modulo~$p$. Moreover, it is a twisted modular form of parallel weight~$(p-1)$. 

On the orthogonal polynomials side we do the following. Let~$P_1,\dots,P_g$ be the twisted quasimodular forms defined in~\eqref{eq:Pj}. Define~$P:=\sum_{j=1}^g\lambda_jP_j$ and the Atkin scalar product on the space of polynomials~$K[J]$
\[
\langle g,h \rangle\:= \text{ the constant term in the Laurent $t$-expansion of }-g\cdot h\frac{J P}{J'}\,\quad g,h\in K[J]\,.
\]
As before, we define the Atkin polynomials~$\{A_n(J)\}_n$ as the family of orthogonal polynomials obtained via the Gram–Schmidt orthonormalization process. 

Observe that if~$h_p$ is  lift of the Hasse invariant for~$Y$, then the~$t$-expansion at~$i\infty$ of its Serre derivative satisfies
\[
\vartheta{h_p}\=Dh_p+(p-1)\frac{\chi(\Gamma)}{2}\sum_{j=1}^g\lambda_jP_j\cdot h_p\;\equiv\; -\frac{\chi(\Gamma)}{2}\cdot P\mod p\,,
\]
since~$h_p(t)\equiv 1\mod p$. This is analogous to the calculation~\eqref{eq:Serrecong} for the partial Hasse invariants. The rest of the proof of Theorem~\ref{thm:main} can be repeated verbatim by replacing~$h_{p,j}$ with the Hasse invariant~$h_p$ and~$\langle\,,\rangle_j$ with the new product~$\langle\,,\rangle$ to prove the following result. 

\begin{theorem}
\label{thm:main2}
Let~$Y\hookrightarrow M_F$ be a non-compact genus-zero curve in a Hilbert modular variety of dimension~$g$ with integral model over~$\mathcal{O}_K[S^{-1}]$. Assume that the family of Atkin orthogonal polynomials~$\{A_{n}(J)\}_n$ exists on~$K[J]$. For every rational prime~$p\not\in S$ there exists~$n_{p}\in\Z_{\ge0}$ such that~$A_{n_p}(J)$ has~$p$-integral coefficients and
\[
A_{n_p}(J)\equiv\mathrm{no}_{p}(J)\mod p\,.
\]
\end{theorem}

The integer~$n_p$ in the statement of Theorem~\ref{thm:main2} is the degree of the non-ordinary polynomial~$\mathrm{no}_p^\mathcal{Y}$. Differently from the integers~$n_{p,j}$ of Theorem~\ref{thm:main}, we do not know how to compute this number in general. The problem is that the divisor of the Hasse invariant~$h_p$ is not reduced, since the points of superspecial reduction, which are located at the intersection of different components of the non-ordinary locus, are counted more than once in general (the number~$n_p$ can be computed using Theorem~\ref{thm:main} though).
So while Theorem~\ref{thm:main2} has the advantage of describing the non-ordinary locus in terms of only one family of orthogonal polynomials, it is less informative in two ways: first it does not describe the intersection of~$\mathcal{X}_{\overline{\fp}}$ with the components of the non-ordinary locus of~$\mathcal{M}_{F,s}$; second, it cannot be directly used for the explicit computation of the non-ordinary locus. 

\section{Partial Hasse invariants from Padé approximants of solutions of Picard–Fuchs differential equations}
\label{sec:PF}

In this section we explain that families of orthogonal polynomials related to the non-ordinary locus of curves in Hilbert modular varieties arise from Picard–Fuchs differential equations. This point of view is more natural and more general; the definition of the Atkin polynomials in terms of modular forms is then a consequence of the modularity of the Picard–Fuchs differential equations. 

We first recall the notion of Padé approximation of a power series near~$\infty$. Let~$f(x)=\sum_{n=0}^\infty{f_nx^{-n-1}}$ be a power series in~$x^{-1}$, and let~$m\ge1$ be an integer. The~\emph{$[m-1,m]$-Padé approximant} of~$f(x)$ is the rational function
\[
R_m(x)\=\frac{\alpha_0+\alpha_1x+\cdots+\alpha_{m-1}x^{m-1}}{\beta_0+\beta_1x+\cdots+\beta_{m-1}x^{m-1}+x^m}
\]
with the characterizing property that~$R_m(x)=f(x)+O(x^{-2m-1})$. Note that we chose to normalize the Padé approximants by requiring that the denominator is a monic polynomial; other choices are possible. Padé approximation is a classical tool in the study of rational approximations; it turns out to be very useful in a variety of situations, and it is implemented in many computer algebra systems, for instance PARI/GP (see the example in Section~\ref{sec:Teich}). 

We re-state Part~2 of Proposition~\ref{prop:ortho} in terms of Padé approximants. 

\begin{lemma}
\label{lem:Pade}
Let~$\langle\,,\rangle$ be a scalar product on~$K[x]$ as in Proposition~\ref{prop:ortho}, and let~$\{P_n(x)\}_n$ be the family of orthogonal polynomials obtained via the Gram–Schmidt process applied to the basis~$\{x^n\}_{n\ge0}$, where~$P_n(x)$ is monic of degree~$n$, and let~$\{Q_n(x)\}_n$ be as in the proposition. Let
\[
\Phi(x)=\sum_{n=0}^\infty{\langle x^n,1\rangle x^{-n-1}}\in K[x^{-1}]
\]
be the generating function of moments of~$\langle\,,\rangle$. Then for every integer~$m\ge1$ the $[m-1,m]$-Padé approximant of~$\Phi(x)$ exists and is~$R_m(x)=\frac{Q_m(x)}{P_m(x)}$. 
\end{lemma} 

It is explained in Section~1 of~\cite{BPf} that to a curve~$Y$ in a $g$-dimensional Hilbert modular variety are attached~$g$ second-order (Picard–Fuchs) differential operators~$L_1,\dots,L_g$. Let~$\mathcal{X}\to\mathcal{Y}$ be the integral model of~$Y$ defined over~$R=\mathcal{O}_K[S^{-1}]$. In~\cite{BPf} it is proven that for every~$j\in\{1,\dots,g\}$ the differential equation~$L_jv=0$ admits a holomorphic solution in~$\mathcal{O}_K[S^{-1}][\![t]\!]$ at any maximal unipotent monodromy point. Moreover, the integral solutions for different~$j\in\{1,\dots,g\}$ are related by congruence relations modulo primes.  

The next theorem essentially states that suitable Padé approximation of the logarithmic derivative of the integral solutions computes the non-ordinary points of $\mathcal{X}_{\overline{\fp}}$ for every~$\fp$ over~$p\notin S$. 

\begin{theorem}
\label{thm:Pade}
Let~$\mathcal{Y}\to\mathcal{X}$ and~$R$ be as above, and let $L_{1},\dots,L_{g}$ be the associated Picard–Fuchs differential operators. For~$j=1,\dots,g$ let $y_{j}(t)\in 1+R[\![t]\!]$ be the normalized integral solution to~$L_{j}v=0$ in~$t=0$. Set $J:=t^{-1}$ and let $\frac{S_{j,m}(J)}{T_{j,m}(J)}$ be the~$[m-1,m]$-Padé approximant of the power series
\begin{equation}
\label{eq:gfPade}
\Phi_{j}(t)\:= t-\frac{2\cdot t^{2}}{\chi(Y)\cdot\lambda_{j}\cdot N}\frac{y_{j}'(t)}{y_{j}(t)}\,,
\end{equation}
where $N$ is the least common multiple of the orders of the elliptic points of $Y$, and $\lambda_{j}$ is the $j$-th Lyapunov exponent~\eqref{eq:Lyap}. Then for every $p\not\in S$ the polynomial~$T_{j,n_{p,j}}(J)$  has $p$-integral coefficients for $n_{p,j}=\deg(\mathrm{ph}_{p,j}(J))$ and it holds
\[
T_{j,n_{p,j}}(J)\;\equiv\;\mathrm{ph}_{p,j}(J)\mod p\,.
\]
\end{theorem}

From Theorem~\ref{thm:Pade} and the classical relation between denominators of Padé approximation and orthogonal polynomials (Lemma~\ref{lem:Pade}) the following corollary follows.

\begin{corollary}
For each $j=1,\dots,g$ there exists a family of orthogonal polynomials~$\{T_{j,n}(J)\}_{n}\subset K[J]$ such that, for every $p\not\in S$ the polynomial~$T_{j,n_{p,j}}(J)$  has $p$-integral coefficients for $n_{p,j}=\deg(\mathrm{ph}_{p,j}(J))$ and satisfies
\[
T_{j,n_{p,j}}(J)\;\equiv\;\mathrm{ph}_{p,j}(J)\mod p\,.
\]
\end{corollary} 

\begin{remark}
\label{rmk:PtoA}
As the corollary suggests, and we show below in detail, Theorem~\ref{thm:Pade} implies Theorem~\ref{thm:main}.
One can recover the definition of Atkin's polynomials in virtue of the modularity of solutions of the Picard–Fuchs differential equations of Kobayashi geodesics (see Section~1 in~\cite{BPf}). Explicitly, modularity means that the holomorphic solution~$y_j(t)$ of the differential equation~$L_jv=0$ can be lifted to a twisted modular form~$f_j(\tau)$ of weight~$(0,\dots,0,1,0,\dots,0)$ with~$1$ in the~$j$-th position. In other words, considering~$t$ as a modular function, the~$t$-expansion of~$f_j(\tau)$ at the cusp where~$t$ has a zero is~$y_j(t)$. In particular it holds for the derivative~$\frac{d}{dt}y_j(t)=\frac{1}{t'}\frac{d}{d\tau}f_j(\tau)$, where~$t'=\frac{d}{d\tau}t$. Assume for simplicity that~$Y\simeq\Po/\Gamma$ is an~$n$-punctured sphere and that~$\Gamma$ is torsion-free, which implies that~$\chi(Y)=2-n$. By using the fact that~$f_j$ has zeros of known order precisely where~$t$ has a cusp (see Section 2 of~\cite{BPf}), one proves the identity 
\begin{equation}
\label{eq:fjdj}
f_j(\tau)^{2\lambda_{j,2}}\cdot t(\tau)^{(n-2)\lambda_{j,1}}\=\Delta_j(\tau)^{(n-2)\lambda_{j,1}}\,,
\end{equation}
where~$\lambda_j=\frac{\lambda_{j,1}}{\lambda_{j,2}}$ and~$\Delta_j$ is as in Remark~\ref{rmk:deltaj}. For~$\Phi_j(t)$ as in Theorem~\ref{thm:Pade} it follows then
\[
\Phi_j(t)\=t\+\frac{2t^2}{(n-2)\lambda_j}\frac{d\log(y_j(t))}{dt}=\frac{t^2}{\lambda_{j,1}(n-2)}\frac{d\log(t^{(n-2)\lambda_{j,1}}\cdot y_j(t)^{2\lambda_{j,2}})}{dt}\,.
\]
By exploiting the modularity of~$y_j(t)$ and the relation~\eqref{eq:fjdj} it follows
\[
\Phi_j(t)\=\frac{t(\tau)^2}{\lambda_{j,1}(n-2)}\frac{d\log(\Delta_j(\tau)^{(n-2)\lambda_{j,1}})}{t'(\tau)\cdot d\tau}\=\frac{t(\tau)^2}{t'(\tau)}\frac{\Delta'_j(\tau)}{\Delta_j(\tau)}\=-\frac{P_j(\tau)}{J'(\tau)}\,,
\]
where~$J(\tau)=\frac{1}{t(\tau)}$. This is the generating function of moments of the Atkin polynomials defined in Section~\ref{sec:main}, and Theorem~\ref{thm:main} follows from Theorem~\ref{thm:Pade}. Nevertheless, Theorem~\ref{thm:Pade} does not assume modularity, and the relation between Padé approximation and non-ordinary locus can in principle be applied to one-parameter families of abelian varieties whose Picard–Fuchs differential equations admit integral solutions, but not necessarily modular. 
\end{remark}

\begin{proof}[Proof of Theorem~\ref{thm:Pade}]
Let $p\not \in S$ be a prime, and let $\fp\subset R$ be a prime above $p$. The theorem can be proven directly by using the modulo~$\fp$ congruences for the integral solutions of the Picard–Fuchs operators~$L_j$ discovered in~\cite{BPf}. However, these congruences are special to curves in Hilbert modular varieties; we propose then a more general argument, that can be adapted to work also in other cases.  

Let~$L$ be a nilpotent second-order Fuchsian differential operator over a characteristic~$p\ne2$ field~$k$. Assume further that the local exponents~$\gamma_1,\dots,\gamma_m,\gamma_\infty$ at the regular singular points~$t_1,\dots,t_m,\infty$ of~$L$ are in~$\mathbb{F}_p$ and that~$Lv=0$ does not have two solutions over~$k[[t]]$ independent over~$k(t^p)$ (see Chapter 9 of~\cite{Dwork} for definitions and details). Then, as written in~\cite{Dwork}, every solution~$\bar{y}$ of~$Lv=0$ is of the form
\begin{equation}
\label{eq:bary}
\bar{y}\= g(t)\prod_{i=1}^m(t-t_i)^{\widetilde{\gamma}_i}\cdot h(t^p)\quad\in k(t)\,,
\end{equation}
where~$g(t)\in k[t]$ is such that~$(t-t_i)\nmid g(t)$ for every~$i\in\{1,\dots,m\}$, and $\widetilde{\gamma}_i\in[0,\dots,p-1]$, and~$h(t^p)\in k(t^p)$. The polynomial~$g(t)\prod_{i=1}^m(t-t_i)^{\widetilde{\gamma}_i}$ is itself a polynomial solution of~$Lv=0$. 

A simple computation shows that the logarithmic derivative~$\bar{y}'(t)/\bar{y}(t)$ is a rational function of~$t$
\begin{equation}
\label{eq:logdp}
\frac{\bar{y}'(t)}{\bar{y}(t)}\=\frac{\bigl(g(t)\prod_{i=1}^m(t-t_i)^{\widetilde{\gamma}_i}\bigl)'}{g(t)\prod_{i=1}^m(t-t_i)^{\widetilde{\gamma}_i}}\=\frac{A(t)}{g(t)\prod_{i=1}^m(t-t_i)^{\epsilon_i}}\,,
\end{equation}
where~$\epsilon_i\ne0$ only if~$\gamma_i\ne0$, and~$\deg\bigl(A(t)\bigr)=n-1$ for $n=\deg\bigl(g(t)\prod_i(t-t_i)^{\epsilon_i}\bigr)$. 

Assume now that the differential operator~$L$ comes from reduction modulo~$\fp$ of a differential operator defined over the ring~$R$ admitting a~$\fp$-integral solution~$y$, and that~$\bar{y}$ is the modulo~$\fp$ reduction of this solution. The same argument used at the end of the proof of Theorem~\ref{thm:main} proves that the denominator of the~$[n-1,n]$-Padé approximant of~$y'/y$ is~$g(t)\prod_{i=1}^m(t-t_i)^{\epsilon_i}$, i.e., the denominator of the right-hand side of~\eqref{eq:logdp}. In particular, the zeros of the denominator of the Padé approximant modulo~$\fp$ are the zeros of the minimal polynomial solution modulo~$\fp$. The point is that this polynomial solution, in the case the differential operator~$L$ comes from geometry, often carries arithmetic information on the underlying family of varieties. We work out the details in the case~$L$ is the reduction of one of the Picard–Fuchs operators of a Kobayashi geodesic. 

In the case of Kobayashi geodesics, as explained in Corollary 2 of~\cite{BPf}, the polynomial solutions of the modulo~$\fp$ reduction of the Picard–Fuchs differential operators are essentially entries of the Hasse–Witt matrix of the associated family of abelian varieties~\eqref{eq:modelp}. Therefore, they carry information on the non-ordinary locus, in a way we make precise below. In this paper we consider partial Hasse polynomials in the variable~$J$ corresponding to a Hauptmodul with a pole at~$i\infty$, which is the reciprocal of the Hauptmodul~$t$ we use in the Picard–Fuchs differential equations, so we do a change of variable first. Let~$\alpha_{\fp,j}(t)$ be the polynomial solution modulo~$\fp$ of the Picard–Fuchs equation~$L_jv=0$ of the form~$\alpha_{\fp,j}(t)=g_j(t)\prod_{i=1}^m(t-t_{j,i})^{\widetilde{\gamma}_{j,i}}$, where $g(t)$ and~$\widetilde{\gamma}_{j,1},\dots,\widetilde{\gamma}_{j,m}$ satisfy the assumptions stated after~\eqref{eq:bary}. Define
\[
\beta_{\fp,j}(J)\:=\alpha_{\fp,j}(J^{-1})\cdot J^{d_{\fp,j}}\,,
\]
where $d_{\fp,j}=\deg(\alpha_{\fp,j}(t))=\frac{\chi(Y)}{2}\cdot N\cdot(\lambda_{j}-p\lambda_{j'})$, as proven in Theorem 3 of~\cite{BPf}. Then it holds (Corollary 2 of~\cite{BPf})
\begin{equation}
\label{eq:beta}
\frac{\beta_{\fp,j}(J)}{\gcd(\beta_{\fp,j}(J),\tfrac{d}{dJ}\beta_{\fp,j}(J))}\= \mathrm{ph}_{\fp,j}(J)\,.
\end{equation}
By using \eqref{eq:logdp} and noticing that $d_{\fp,j}\equiv \frac{\chi(Y)}{2}\lambda_{j}N\mod\fp$, one computes from the definition of $\Phi_{j}(J)$ that
\[
\Phi_{j}(t)\=t-\frac{2t^{2}}{\chi(Y)\lambda_{j}N}\frac{y'_{j}(t)}{y_{j}(t)}\;\equiv\; t -\frac{2t^{2}}{\chi(Y)\lambda_{j}N}\frac{\alpha'_{\fp,j}(t)}{\alpha_{\fp,j}(t)}\;\equiv\;\frac{2\bigl(d_{\fp,j}\cdot t\alpha_{\fp,j}(t)-t^{2}\alpha'_{\fp,j}(t)\bigr)}{\chi(Y)\lambda_{j}N\cdot\alpha_{\fp,j}(t)}\mod\fp\,.
\]
From the definition of $\beta_{\fp,j}(J)$, the relation $J=t^{-1}$, and Equation~\eqref{eq:beta}, it finally follows that
\[
\Phi_{j}(J)\;\equiv\; \frac{2}{\chi(Y)\lambda_{j}N}\frac{\tfrac{d}{dJ}\beta_{\fp,j}(J)}{\beta_{\fp,j}(J)}\=\frac{R_{\fp,j}(J)}{\mathrm{ph}_{\fp,j}(J)}\mod\fp\,,
\]
for some polynomial $R_{\fp,j}(J)$ of degree $n_{\fp,j}-1$ and coprime with $\mathrm{ph}_{\fp,j}(J)$ by construction.  

The $[n_{\fp,j}-1,n_{\fp,j}]$-Padé approximation of $\Phi_{j}(J)$ gives then
\[
\frac{S_{j,n_{\fp,j}}(J)}{T_{j,n_{\fp,j}}(J)}\=\Phi_{j}(J)\+O(J^{-2n_{\fp,j}-1})\;\equiv\;\frac{R_{\fp,j}(J)}{\mathrm{ph}_{\fp,j}(J)}\mod\fp\,.
\]
The same argument used at the end of the proof of Theorem~\ref{thm:main} proves that~$T_{j,n_{\fp,j}}(J)$ has~$\fp$-integral coefficients and that 
\[
T_{j,n_{\fp,j}}(J)\;\equiv\;\mathrm{ph}_{\fp,j}(J)\mod \fp\,.
\]
\end{proof}

\section{Examples}
\label{sec:examples}

\subsection{Arithmetic cases}

The case of~$\SL_2(\Z)$ (the original Atkin–Kaneko–Zagier result~\cite{KZ}) can be recovered from our result by considering the modular embedding of~$\Po/\SL_2(\Z)$ in a Hilbert modular surface via the diagonal embedding~$\tau\mapsto(\tau,\tau)$. Then~$\varphi_2(\tau)=\tau$ and~$B_2=1$, and the two scalar products~$\langle\,,\rangle_j$ for~$j=1,2$ are the same. This holds because in this case we consider elliptic curves, and there is only one Hasse invariant. 

Similarly, the case of genus-zero subgroups of~$\SL_2(\Z)$ (see for instance~\cite{Sa} and~\cite{Ts}) can be derived from our results by considering the curves~$\Po/\Gamma_0(N)$ as Hirzebruch–Zagier curves in a Hilbert modular surface.

\subsection{Affine triangle curves~$\Po/\Delta(n,m,\infty)$}
\label{sec:tri}

Affine triangle curves have a modular embedding in Hilbert modular varieties~\cite{CW}. We determine explicitly the recursion for their families of Atkin's polynomials. 

Twisted modular forms for the triangle groups~$\Delta(n,m,\infty)$, with~$n,m\ge2$ and~$n+m\ge5$ (hyperbolicity condition) were studied in~\cite{BN}. Let~$\Po/\Delta(n,m,\infty)$ have an embedding in a Hilbert modular variety of dimension~$g$. Then for~$j=1,\dots,g$ there exist non-zero integers~$k_j\le n$ and~$r_j\le m$ such that, if~$N_j:=mn-nr_j-mk_j$, then the integral solutions of the Picard–Fuchs equations arising from~$\Po/\Delta(n,m,\infty)$ are classical Gauss hypergeometric functions
\begin{equation}
\label{eq:hgj}
y_j(t)\={}_2F_1\Bigl(\frac{N_j}{2nm},\frac{N_j+2nr_j}{2nm},1;t\Bigr)\,,
\end{equation}
where
\[
{}_2F_1(a,b,c;t)\:=\sum_{i=0}^\infty{\frac{(a)_i(b)_i}{(c)_i i!}t^i}\,,\quad (x)_i=x\cdot(x+1)\cdots(x+i-1)\,.
\] 
From Corollary 3.3 of~\cite{BN} it follows, with the notation of this paper, that
\[
y_j(t)^{2nm}\cdot t^{N_j}\=\Delta_j^{N_j}\,.
\]
A calculation similar to the one in Remark~\ref{rmk:PtoA} prove that, for the generating function of moments~$\Phi_j(t)$ attached to the Atkin scalar products on~$\Delta(n,m,\infty)$, it holds
\begin{equation}
\label{eq:logPhi}
\Phi_j(t)\=\frac{-t}{N_j}\biggl(\frac{N_jy_j(t)+2nm\cdot t\cdot y_j'(t)}{y_j(t)}\biggr)\,.
\end{equation}
Gauss's contiguous relations and~\eqref{eq:hgj} permit to express~$t\cdot y_j'(t)$ in terms of contiguous hypergeometric series; it follows that
\[
\Phi_j(t)\=\frac{t\cdot{}_2F_1\Bigl(\frac{N_j}{2nm}+1,\frac{N_j+2nr_j}{2nm},1;t\Bigr)}{{}_2F_1\Bigl(\frac{N_j}{2nm},\frac{N_j+2nr_j}{2nm},1;t\Bigr)}\,.
\]
Gauss also famously described the continued fraction expansion of a quotient of contiguous hypergeometric functions~\cite{Gauss}:
\[
\frac{{}_2F_1(a+1,b;c;z)}{{}_2F_1(a,b;c;z)}\=\frac{1}{1-\frac{\lambda_1z}{1-\frac{\lambda_2z}{1-\ddots}}}\,
\]
with~$\lambda_1=b/c$ and
\[
\lambda_{2k}\=\frac{(a+k)(c-b+k-1)}{(c+2k-2)(c+2k-1)}\,,\quad\lambda_{2k+1}\=\frac{(b+k)(c-a-1+k)}{(c+2k-1)(c+2k)}\,.
\]
By Part 3 of Proposition~\ref{prop:ortho}, the recursion formula for the families of Atkin polynomials is determined by the above values~$\lambda_k$. The final result is stated in the next proposition.

\begin{proposition}
\label{prop:rect}
Let~$\Delta(n,m,\infty)$ be a hyperbolic triangle group and consider the modular embedding~$\Po/\Delta(n,m,\infty)\hookrightarrow M_F$ into a Hilbert modular variety of dimension~$g$. Let~$N_j$ and~$r_j$ be as in~\eqref{eq:hgj}. 
For~$j=1,\dots,g$, the~$j$-th family of Atkin polynomials~$\{A_{j,k}\}_k$ attached to~$\Delta(n,m,\infty)$ satisfy the recursion
\[
A_{j,k+1}(J)\=(J-a_{j,k})\cdot A_{j,k}(J)-b_{j,k}\cdot A_{j,k-1}(J)\,,\quad k\ge1
\]
for
\[
a_{j,k}\=\frac{2k^2-(2N_j^2-\frac{2N_jr_j}{m}-\frac{r_j}{m})}{4k^2-1}\,,\quad b_{j,k}\=\frac{(k+N_j)(k-N_j-1)(k-N_j-\frac{r_j}{m})(k+N_j+\frac{r_j}{m}-1)}{k(k-1)(2k-1)^2}
\]
and initial datum~$A_{j,0}(J)=1$ and~$A_{j,1}(J)=J-\frac{N_j+nr_j}{2mn}$.
\end{proposition}

\begin{proof}
Only the values of the initial data are left to be proven. From the definition of~$P_n(x)$ in~\eqref{eq:GSp} it follows immediately that~$A_{j,0}(J)=1$ and that~$A_{j,1}(J)=J-\langle J,1\rangle_j$. The constant term of~$A_{j,1}(J)$ is nothing but the coefficient of~$J^{-1}$ in the generating function of moments~$\Phi_j(J)$, or, equivalently, the coefficient of~$t^2$ in~$\Phi_j(t)$. By using the expression in~\eqref{eq:logPhi} one sees that~$\langle J,1\rangle_j$ is given by~$y_{j,1}\cdot\frac{2mn}{N_j}$, where~$y_j(t)=\sum_{l=0}^\infty{y_{j,l}t^l}$ is as in~\eqref{eq:hgj}. The explicit form of the hypergeometric solution implies that~$y_{j,1}=\frac{N_j}{2nm}\cdot\frac{N_j+2r_jn}{2nm}$ and then the value of~$A_{j,1}(J)$.
\end{proof}

\subsubsection{The curve $\Po/\Delta(2,5,\infty)$}
\label{sec:Lane}
In his Ph.D. thesis~\cite{Lane}, M. Lane constructed Atkin's polynomials for the family of Hecke triangle groups~$\Delta(2,m,\infty)$ for~$m\ge3$, and related their zeros to the non-ordinary locus of certain explicit families of hyperelliptic curves, by using explicit calculations of their Hasse–Witt matrices. 

As written in Section 2.2 of~\cite{Lane}, the Hecke triangle group~$\Delta(2,m,\infty)$ admit a modular embedding in a Hilbert modular variety of  dimension~$g=\frac{\varphi(m)}{2}$, where~$\varphi$ is the Euler's totient function. For~$j=1,\dots,\varphi(m)/2$, the specialization to~$N_j=m-2j$ (obtained by setting $n=2$ and~$k_j=1$ and~$r_j=j$ in the definition of~$N_j)$ of the formulae for~$a_{j,k}$ and~$b_{j,k}$ in Proposition~\ref{prop:rect} give
\[
\begin{aligned}
a_{j,k}&\=\frac{m^2(16k^2-1)-8r_jm+4j^2}{8m^2(4k^2-1)}\,,\\
b_{j,k}&\=\frac{(4mk-3m+2j)(4mk-5+2j)(4mk+m-2j)(4mk-m-2j)}{4^5m^4k(k-1)(2k-1)}\,.
\end{aligned}
\]
This recovers Part (i) of Theorem~3 of~\cite{Lane}. 

We consider in detail the case of the non-arithmetic curve~$\Po/\Delta(2,5,\infty)$ in the Hilbert modular surface of discriminant~$D=5$, compute the supersingular locus for some primes, and compare it to the work of Lane. For~$j\in\{1,2\}$, the Atkin polynomials~$\{A_{k,j}(J)\}_k$ satisfy the recursion in Proposition~\ref{prop:rect} with 
\[
a_{k,j}\=\frac{25(16k^2-1)-40j+4j^2}{200(4k^2-1)}\,,\quad b_{k,j}\=\frac{(20k-15+2j)(20k-25+2j)((20k-2j)^2-25)}{640000\cdot k(k-1)(2k-1)^2}\,.
\]
with initial data~$A_{0,j}(J)=1$ and $A_{1,j}(J)=J-\frac{5+2j}{20}$. 

We compute the first polynomials for~$j=1$
\[
A_{1,1}(J)=J-\frac{7}{20}\,,\quad A_{2,1}(J)=J^2-\frac{549}{600}J+\frac{3213}{48000}\,,
\]
and for~$j=2$
\[
A_{1,2}(J)=J-\frac{9}{20}\,,\quad A_{2,2}(J)=J^2-\frac{581}{600}J+\frac{1653}{16000}\,,\quad A_{3,2}(J)=J^3-\frac{184}{125}J^2+\frac{3393523}{6400000}J-\frac{3158883}{128000000}\,.
\]
Lane relates the zeros of the above Atkin polynomials to the non-ordinary locus of the family of hyperelliptic curves
\[
y^2\=x^{11}-2ax^6+bx\,,\quad a,b\in\overline{\F}_p\,,\quad b-a^2\ne0\,.
\]
By considering instead~$\Po/\Delta(2,5,\infty)$ embedded in the Hilbert modular surface of discriminant~$D=5$, the geometric object associated to the Atkin polynomials is the family of Jacobians of
\begin{equation}
\label{eq:delta5}
y^2=\begin{cases}
x^5-5x^3+5x-2\eta & \eta\ne\infty\,,\\
x^5-1 & \eta=\infty
\end{cases}\,.
\end{equation}
The relation between the Hauptmodul~$J$ of~$\Delta(2,5,\infty)$ and the parameter~$\eta$ in~\eqref{eq:delta5} is~$J=(\eta^2-1)^{-1}$. Let~$p\ge7$ be a prime. It is proven in~\cite{BLpHi} that, keeping the notation of Theorem~\ref{thm:main}, if~$p$ is inert in~$\Q(\sqrt{5})$, then
\[
n_{p,1}\=\frac{p-3}{20}+\frac{\epsilon}{2}+\frac{5-\delta_1}{5}\,,\quad n_{p,2}=\frac{3p-1}{20}+\frac{\epsilon}{2}+\frac{5-\delta_2}{5}\,,
\]
where~$\epsilon=1$ if~$p\equiv1\mod4$ and~$\epsilon=0$ otherwise, and~$(\delta_1,\delta_2)=(1,5)$ if~$p\equiv2\mod5$ and~$(\delta_1,\delta_2)=(5,2)$ if~$p\equiv3\mod5$. If~$p$ is split in~$\Q(\sqrt{5})$ then
\[
n_{p,1}\=\frac{3(p-1)}{20}+\frac{\epsilon}{2}+\frac{5-\delta_1}{5}\,,\quad n_{p,2}=\frac{p-1}{20}+\frac{\epsilon}{2}-\frac{5-\delta_2}{5}\,,
\]
where~$\epsilon=0$ if~$p\equiv1\mod4$ and~$\epsilon=1$ otherwise, and~$(\delta_1,\delta_2)=(5,5)$ if~$p\equiv1\mod5$ and~$(\delta_1,\delta_2)=(1,2)$ if~$p\equiv4\mod5$.

For instance, $p=7$ is inert in~$\Q(\sqrt{5})$ and we get~$n_{7,1}=1$ and~$n_{7,2}=1$. Then
\[
\mathrm{ss}_{7}(J)\equiv\mathrm{lcm}(J-7/20,J-9/20)\mod 7\equiv J(J-5)\mod 7
\]
is the supersingular polynomial (and there are no superspecial values). This means that, for~$p=7$, the supersingular curves in~\eqref{eq:delta5} whose Jacobian is supersingular are those with parameter~$\eta\in\{1,2\}$.

For the split prime~$p=11$ we have~$n_{11,1}=2$ and~$n_{11,2}=1$ and then
\[
\begin{aligned}
\mathrm{no}_{11}(J)&\equiv\mathrm{lcm}(A_{2,1}(J),A_{1,2}(J))\equiv (J-1)(J-8)\mod 11\,,\\
\mathrm{sp}_{11}(J)&\equiv\mathrm{gcd}(A_{2,1}(J),A_{1,2}(J))\equiv (J-1)\mod 11\,.
\end{aligned}
\]

Finally, for~$p=13$ we have~$n_{13,1}=1$ and~$n_{13,2}=3$ and then
\[
\begin{aligned}
\mathrm{ss}_{13}(J)&\equiv\mathrm{lcm}(A_{1,1}(J),A_{3,2}(J))\equiv J(J-1)(J-9)\mod 13\,,\\
\mathrm{sp}_{13}(J)&\equiv\mathrm{gcd}(A_{1,1}(J),A_{3,2}(J))\equiv (J-1)\mod 13\,.
\end{aligned}
\]

\subsection{The Teichmüller curve $W_{17}$}
\label{sec:Teich}

Let~$\alpha=\alpha_\infty=\frac{1+\sqrt{17}}{2}$, and denote by~$[a,b]=a+b\cdot\alpha$ an element of~$\mathcal{O}_{17}$, the ring of integers of the real quadratic field~$\Q(\sqrt{17})$. Consider the following 1-parameter family of genus-two curves:
\begin{equation}
\label{eq:W17}
Y^2\=(X+(At+B))\cdot(X+(Bt+A))\cdot(X^3+C(t+1)X^2+(D(t+1)^2+Et)X+F(t+1)^3+Gt(t+1))\,,
\end{equation}
where
\[
\begin{aligned}
A&\=5[2,1]\,,\quad B\=-2[5,3]\,,\quad C\=[3,1]\,,\quad D\=-\frac{1}{4}[827,529]\\
E&\=2^4\cdot17[3,2]\,,\quad F\=-\frac{1}{2}[4597,2943]\,,\quad G=2\cdot17[271,173]\,.
\end{aligned}
\]
The family of Jacobians of~\eqref{eq:W17} defines a curve~$\Pi_{17}$ in the Hilbert modular surface $\Po^2/\SL(\mathcal{O}_{17}\oplus\mathcal{O}_{17}^\vee)$ of discriminant~$17$. It is a double cover of a component of the (algebraically primitive) Teichmüller curve~$W_{17}$. The family has two associated Picard–Fuchs differential operators~$L_1$ and~$L_2$, described explicitly in Section 8 of~\cite{BM}. The uniformizing group of~$\Pi_{17}$, which is conjugated to the monodromy group of~$L_1v=0$, is a non-arithmetic torsion-free genus-zero Fuchsian group, not commensurable with a triangle group. In particular, the curve~$\Pi_{17}$ is not related to the curves studied in Section~\ref{sec:tri}.

As shown first in~\cite{BM}, the differential equation~$L_jv=0$, for~$j\in\{1,2\}$, has a holomorphic solution~$y_j(t)$ at the regular singular point~$t=0$ with~$y_j(0)=1$ and coefficients in~$\mathcal{O}_{17}[2^{-1}]$. Finally, as proven in Section 10 of the same paper, the curve~$\Pi_{17}$ is defined over~$\mathcal{O}_{17}[34^{-1}]$. 

From now on we consider a prime number~$p\ne2,17$. Corollary 2 of~\cite{BLpHi} and the dimension formula for the space of twisted modular forms on curves in Hilbert modular surfaces (Theorem 3.2 of~\cite{MZ}) imply that
\[
\mathrm{deg}(\mathrm{ph}_{p,1}(t))=\begin{cases}\frac{p-3}{2} &\text{if $p$ is inert in $\Q(\sqrt{17})$}\\\frac{3(p-1)}{2}& \text{if $p$ is split in $\Q(\sqrt{17})$}\end{cases},\quad
\mathrm{deg}(\mathrm{ph}_{p,2}(t))=\begin{cases}\frac{3p-1}{2} &\text{if $p$ is inert in $\Q(\sqrt{17})$}\\\frac{p-1}{2}& \text{if $p$ is split in $\Q(\sqrt{17})$}\end{cases}.
\]
In particular, the truncation of~$y_j(t)$ at order~$p$ does not give the partial Hasse polynomial if~$j=1$ and~$p$ is split in~$\Q(\sqrt{17})$ or if~$j=2$ and~$p$ is inert, because the degree of these polynomials is larger than~$p-1$ (see Corollary~3 in~\cite{BPf}). Nevertheless, Theorem~\ref{thm:Pade} applies to compute both the Atkin polynomials and all partial Hasse polynomials via Padé approximation. The generating functions of moments in~$J=t^{-1}$, computed from the integral solutions of the Picard-Fuchs equations, start by
\[
\begin{aligned}
\Phi_1(J)&\=1\+\frac{27-5\sqrt{17}}{8}J^{-1}\+\frac{4729-1135\sqrt{17}}{64}J^{-2}\+\frac{2069771-501805\sqrt{17}}{1024}J^{-3}\+\cdots\,,\\
\Phi_2(J)&\=1+\frac{23-5\sqrt{17}}{4}J^{-1}\+\frac{4607-1113\sqrt{17}}{32}J^{-2}\+\frac{518341-125699\sqrt{17}}{128}J^{-3}\+\cdots\,.
\end{aligned}
\]
The denominators of the first~$[m-1,m]$-Padé approximants, that is, the first Atkin polynomials, are listed below. The coefficients are written in terms of the integral basis~$[1,\alpha]$. For~$j=1$ one has
\[
\begin{aligned}
A_{1,1}(J)&\=J-\frac{[-16,5]}{2^2}\,,\quad A_{1,2}(J)\=J^2+\frac{[-5484,2013]}{2^4\cdot19}J+\frac{[16588,-6433]}{2^6\cdot19}\,,\\
A_{1,3}(J)&\=J^3+\frac{[-1778264,651953]}{2^2\cdot43\cdot373}J^2+\frac{[136493368,-53143951]}{2^5\cdot43\cdot373}J+\frac{[-914898692,357094875]}{2^9\cdot43\cdot373},
\end{aligned}
\]
and for~$j=2$ 
\[
\begin{aligned}
A_{2,1}(J)&\=J-\frac{[-14,5]}{2}\,,\quad A_{2,2}(J)\=J^2+\frac{[-41680,15355]}{2^2\cdot557}J+\frac{[188220,-73249]}{2^4\cdot557}\,,\\
A_{2,3}(J)&\=J^3+\frac{[-34055908888,12515409651]}{2^2\cdot103\cdot2939707}J^2+\frac{[10749417998972,-4186567148401]}{2^7\cdot103\cdot2939707}J\\&\+\frac{[-11766694843580,4593116159421]}{2^8\cdot103\cdot2939707}\,.
\end{aligned}
\]

The primes~$p=3,5$ are inert in~$\Q(\sqrt{17})$. As recalled in Section~\ref{sec:Koba}, the non-ordinary and the supersingular locus of~$\Pi_{17}$ modulo~$p$ coincide. By using the above formulae for the degree of the partial Hasse polynomials and Theorem~\ref{thm:main}, the abelian surfaces of supersingular reduction over~$\Pi_{17}$ correspond to the zeros of the polynomials:
\begin{itemize}[wide=0pt]
\item for~$p=3$ 
\[
\mathrm{ph}_{3,1}(J)\equiv A_{1,0}(J)\equiv1\mod 3,\quad \mathrm{ph}_{3,2}(J)\equiv A_{2,4}(J)\equiv J^4+(1+2\sqrt{17})\cdot(J^3+J^2+J) +1 \mod 3\,;
\]
\item for~$p=5$ 
\[
\begin{aligned}
\mathrm{ph}_{5,1}(J)&\equiv A_{1,1}(J)\equiv J+1\mod5\,,\\
\mathrm{ph}_{5,2}(J)&\equiv A_{2,7}(J)\equiv (J+1)\cdot(J^6+(2+4\sqrt{17})\cdot(J^4+J^2)-3\sqrt{17}\cdot J^3+1)\mod 5\,.
\end{aligned}
\]
In particular, the Jacobian of the curve in~\eqref{eq:W17} corresponding to~$t=-1$ is superspecial modulo~$p=5$.
\end{itemize}

As written in the last paragraph of the introduction, the partial Hasse polynomials~$\mathrm{ph}_{p,j}$ for~$\Pi_{17}$ are palindromic for all primes~$p<100$. This is not a feature of the Atkin polynomial~$A_{j,n}(J)$ (over~$\mathcal{O}_{17}[2^{-1}][J]$), nor of its modulo~$p$ reduction for~$n\ne n_{j,p}$.
\bibliography{Atkin}{} 
\bibliographystyle{plain}
\end{document}